\documentclass[11p]{article}

\usepackage{times}
\usepackage{amsthm}
\usepackage{amsmath}
\usepackage{amssymb}
\usepackage{mathrsfs}
\usepackage{pb-diagram}
\usepackage{xcolor}
\usepackage{hyperref}

\def\th{\theta}
\def\Th{\Theta}

\def\e{{\sf e}}

\def\r{{\rm r}}
\def\d{{\rm d}}
\def\bu{\bullet}

\def\({\left(}
\def\[{\left[}
\def\){\right)}
\def\]{\right]}

\def\si{\sigma}
\def\Si{\Sigma}
\def\G{\mathcal G}

\def\<{\langle}
\def\>{\rangle}

\usepackage{tikz-cd} 

 \newtheorem{thm}{Theorem}[section]
 \newtheorem{cor}[thm]{Corollary}
 \newtheorem{lem}[thm]{Lemma}
 \newtheorem{prop}[thm]{Proposition}
 \theoremstyle{definition}
 \newtheorem{defn}[thm]{Definition}
 \theoremstyle{remark}
 \newtheorem{rem}[thm]{Remark}
 \newtheorem{ex}[thm]{Example}
 \numberwithin{equation}{section}

\numberwithin{equation}{section}

\begin{document}


\title{Morphisms of Groupoid Actions and Recurrence}


\author{F. Flores and M. M\u antoiu
}

\maketitle


\begin{abstract}
Topological groupoids admit various types of morphisms. We push these notions to the level of continuous groupoid actions to obtain various types of groupoid action morphisms. Some dynamical properties and their relation to these morphisms are studied. Among them are recurrence, various forms of transitivity, minimality, limit, recurrent, periodic and almost periodic  points.
\end{abstract}

\tableofcontents

\section{Introduction}\label{introduction}

In a recent paper \cite{FM} we studied some basic concepts of topological dynamics in the setting of continuous groupoid actions \cite{Wi}, focusing on recurrence phenomena. Guided by classical textbooks on group dynamical systems, as \cite{Au,dV,Gl}, we examined notions as (point, topological or recurrent) transitivity, minimality, (non-)wandering, fixed, recurrent, periodic and almost periodic points. Various examples have been considered and we exhibited some, maybe unexpected, phenomena appearing for groupoids for which the source map is not open.

\smallskip
In \cite[Sect.7]{FM} we investigated how epimorphisms preserve these notions. Epimorphism meant {\it one} topological groupoid acting on two topological spaces and a continuous equivariant surjection connecting them. After a review of some basic definitions from the realm of topological groupoid actions (Subsections \ref{gogam} and \ref{sucritor}), we extend here the study of morphisms in two new directions, the main purpose being to {\it allow a change of groupoids}. There are several ways to make this change, depending on what "groupoid morphism" means.

\smallskip
As groupoids are special classes of categories, the first (natural) notion of groupoid morphism is, of course, a continuous functor. This is what is taken into account in most of the references. We used this type of morphism to introduce, in Subsection \ref{goam}, morphisms of actions, consisting of pairs of maps
$$
(\Psi,f):(\Xi,\Si)\to\big(\Xi',\Si'\big)\,,
$$ 
where $\Psi:\Xi\to\Xi'$ is a continuous functor and $f:\Si\to\Si'$ is a continuous map between the topological spaces. A natural compatibility relation (equivariance) is required. Since generalized vague morphisms treated in Section \ref{supritor} extend this, we only point out a single (simple) new result, Proposition \ref{label}, describing the way  recurrence sets are transformed under these morphisms. It will be useful in subsequent sections.

\smallskip
As an extension, based on the notions of {\it vague functor} and {\it vague isomorphism} (see \cite{MZ}), 
we proceeded to introduce the notion of {\it generalized vague morphisms}. This construction is no longer based on a direct function between the groupoids $\Xi,\Xi'$ but on a morphism between the pullbacks $\Xi(\pi,A),\Xi'(\pi',A')$. In this way, a weaker and indirect notion of morphism is achieved (if a morphism $\Gamma:\Xi\to\Xi'$ exists, it induces naturally a generalized vague morphism, for every pair of pullbacks $\Xi(\pi,A),\Xi'(\pi',A')$). We push this to the level of groupoid actions by adding a continuous function $h:\Si\to \Si'$ with a suitable compatiblity condition.

\smallskip
The last type of morphism we introduce is based on {\it algebraic morphisms}. When studyind groupoid $C^*$-algebras, experts remarked that usual groupoid morphisms do not fit well, and this can be seen at the level of very simple extreme particular cases, groups and spaces. Group algebras $C^*(\G),C^*(\G')$ behave well with respect to (direct) group morphisms $\G\to\G'$, while the Abelian $C^*$-algebras $C_0(X),C_0(X')$ prefer (proper continuous) maps $X\leftarrow X'$ acting in the opposite direction. Recall that the Gelfand functor is contravariant!

\smallskip
References as \cite{Za1,Za2,BS,Bu,BEM,MZ} contributed to a solution, introducing {\it algebraic morphisms} \cite{BEM} (also called {\it actors} in \cite{MZ}), insuring a functorial groupoid $C^*$-algebra construction. Instead of a mapping, such an algebraic morphism $\Phi:\Xi\rightsquigarrow\Xi'$ is a continuous action of $\Xi$ on the topological space $\Xi'$, commuting with the right action of the groupoid $\Xi'$ on itself. Buss, Exel and Meyer argued that "algebraic morphisms are exactly the same as functors between the categories of actions that do not change the underlying space". We refer to \cite[Th.4.12]{BEM} for a precise statement.

\smallskip
It is the purpose of the present article to extend and explore all these generalized notions of morphisms of groupoid actions and to describe their effect on several dynamical properties. The notion of recurrence (dwelling) set is central. Let us now describe in more detail the content of the paper.

\smallskip
{\bf Section \ref{goram}} deals with the definitions of groupoids, actions, orbits and recurrence sets. It is mostly based on our previous work \cite{FM} and classical textbooks in dynamics, such as \cite{Au,dV,Gl}. Subsection \ref{sucritor} reviews concepts as transitivity (of various types), minimality, (non-)wandering sets, periodic, almost periodic, weakly periodic and fixed points. 

\smallskip
On the other hand, in Definition \ref{sprevest} we introduce the concept of groupoid morphism of actions, based on groupoid morphisms (functors). Examples are presented in \ref{startlet}, \ref{novotel} and \ref{homoconstruction}. This type of morphisms make the class of continuous groupoid actions a category. This section ends with the promised result on recurrent sets (Proposition \ref{label}).

\smallskip
{\bf Section \ref{supritor}.} For defining generalized vague morphisms, we need to define the pullback $\Xi(\pi,A)$\,. This needs a continous surjection $\pi:A\to X$ from the topological space $A$ to the unit space $X$ of $\Xi$ and Definition \ref{labar} explains how. Proposition \ref{laker} shows that an action $\Th$ of $\Xi$ induces canonically an action of $\Xi(\pi,A)$, paving the way for Definition \ref{meyer}, which introduces generalized vague morphisms (of groupoids) and Definition \ref{memeyer} which lifts this at the level of actions. Example \ref{furnal} shows that usual morphisms of actions (Definition \ref{sprevest}) are contained, as very particular cases.

\smallskip
Let us consider two actions of two different groupoids $\Xi,\Xi'$, for which a generalized vague morphisms of actions is defined. Even through there is no map $\Xi\to\Xi'$, we can relate the recurrence sets of the two actions by Theorem \ref{both} or by Remark \ref{stift} (which is nothing more than a reformulation). Corollary \ref{securinta} easily follows and relates recurrent transitivity in both actions. Theorem \ref{color} treats the relation between orbits and invariant subsets and Corollaries \ref{secinta}, \ref{garbanzos}, \ref{gogonata} and \ref{rolar} show that (various kinds of) topological transitivity, minimality, periodic points and almost periodic points are preserved (under certain assumptions). Proposition \ref{rollar} provides some different conditions for preserving periodic and almost periodic points, trading hypothesis over the topology of the groupoids for hypothesis over the maps. Finally, Proposition \ref{caciu} treats limit and non-wandering points.

\smallskip
{\bf Section \ref{sucitorr}.} After defining algebraic morphisms $\Phi:\Xi\rightsquigarrow\Xi'$ and giving some (known) examples, we use them to implement changes of groupoids in groupoid actions. This is done in Definition \ref{morfalg}, were we introduced the new notion of {\it algebraic morphism of a groupoid action}. Example \ref{terminalg} shows that it reduces to the usual morphism of group actions when the acting groupoids are groups. Before jumping to the study of the dynamical properties, in Definition \ref{composition} we define a composition of algebraic morphisms of actions and prove its correctness in Proposition \ref{saspermam}. With these morphisms, continuous actions of groupoids form a category. We then prove a Lemma (\ref{constant}), useful later.

\smallskip
As in the case of generalized vague morphisms, there is no direct map from $\Xi$ to $\Xi'$. But once again, and in a different way from that of Theorem \ref{both}, there is a way to connect recurrence sets; this is the content of Theorem \ref{liema}. We exploit this theorem for relating the dynamical properties of two systems joined by an algebraic morphism of actions: Proposition \ref{jnitzel} relates the orbits and the invariant sets, corollaries \ref{sentinta} and \ref{sentintaa} relate the different notions of transitivity, Corollary \ref{siaia} relates the periodic points of the systems and Corollary \ref{siaia2} does the same with almost periodic points. In Definition \ref{amu} we introduce a suitable notion of {\it properness} for algebraic morphisms and we use it for proving Proposition \ref{transflim}, which uses properness for transporting limit points to limit points, recurrent points to recurrent points and weakly periodic points to weakly periodic points.

\smallskip
{\bf Section \ref{zgarier}.} In this appendix we included three facts that we found interesting about algebraic morphisms. Proposition \ref{miraj} is an extension of the result given in \cite[Remark\;2.4, Prop.\;2.12]{Bu}. It proves that if a certain structure map of the algebraic morphism is a homeomorphism, the notion reduces precisely to the morphisms of actions of Definition \ref{sprevest}. In this case, the image of the morphism coincides with the saturation of the unit space of the second groupoid, hence we think the last set as a generalization of the concept of image. Proposition \ref{structure} studies this set and some of its algebraic properties. Finally, in Proposition \ref{enfin}, we prove that algebraic morphisms preserve openness.

\section{Review of groupoid actions}\label{goram}

\subsection{Groupoids and groupoid actions}\label{gogam}

The groupoids $\Xi$ are small categories in which all the morphisms (arrows) are invertible. The objects, also called units, are seen as a particular family $\Xi^{(0)}\!\equiv X\subset\Xi$ of morphisms.
The source and range maps are denoted by ${\rm d,r}:\Xi\to \Xi^{(0)}$ and the family of composable pairs by $\,\Xi^{(2)}\!\subset\Xi\times\Xi$\,. For $M,N\subset X$ one has the important subsets 
\begin{equation}\label{faneaka}
\Xi_M\!:={\rm d}^{-1}(M)\,,\quad\Xi^N\!:={\rm r}^{-1}(M)\,,\quad\Xi_M^N:=\Xi_M\cap\Xi^N.
\end{equation}
Particular cases are the $\d$-fibre $\Xi_x\equiv\Xi_{\{x\}}$\,, the  ${\rm r}$-fibre $\Xi^x\equiv\Xi^{\{x\}}$ and the isotropy group  $\Xi_x^x\equiv\Xi_{\{x\}}^{\{x\}}$ of a unit $x\in X$. 

\smallskip
A topological groupoid is a groupoid $\Xi$ with a topology such that the inversion $\xi\mapsto\xi^{-1}$ and multiplication $(\xi,\eta)\mapsto \xi\eta$ are continuous. The topology in $\Xi^{(2)}$ is the topology induced by the product topology. Whenever the map $\d:\Xi\to X$ is open, we say that {\it the groupoid $\Xi$ is open}. Equivalent conditions are: (i) $\r:\Xi\to X$ is open and (ii) the multiplication is an open map. It is known that a locally compact groupoid possessing a Haar system is open. In particular, \'etale groupoids and Lie groupoids are open.

\smallskip
An equivalence relation on $X$ is defined by $x\approx y$ if $x=\r(\xi)$ and $y={\rm d}(\xi)$ for some $\xi\in\Xi$\,. This leads to the usual notions of {\it orbit, invariant (saturated) set, saturation, transitivity}, etc. 

\begin{defn}\label{grupact}
{\it A (left) groupoid action} is a 4-tuple $(\Xi,\rho,\th,\Si)$ consisting of a groupoid $\Xi$\,, a set $\Si$\,, a surjective map $\rho:\Si\rightarrow X$ ({\it the anchor}) and the action map
\begin{equation*}\label{ganchor}
\th:\Xi\!\Join\!\Si:=\{(\xi,\si)\!\mid\!\d(\xi)=\rho(\si)\}\ni(\xi,\si)\mapsto{\th_\xi(\si)\equiv\xi\!\bu_\th\!\si}\in\Si
\end{equation*} 
satisfying the axioms: 
\begin{enumerate}
\item $\rho(\si)\!\bu_\th\!\si=\si,\, \forall \,\si\in \Si$\,,
\item if $(\xi,\eta)\in \Xi^{(2)}$ and $(\eta,\si)\in \Xi\!\Join\!\Si$, then $(\xi,\eta\!\bu_\th\!\si)\in \Xi\!\Join\!\Si$ and $(\xi\eta)\!\bu_\th\!\si=\xi\!\bu_\th\!(\eta\!\bu_\th\!\si)$\,.
\end{enumerate} 
{\it An action of a topological groupoid} in a topological space (also called {\it a groupoid dynamical system}) is just an action $(\Xi,\rho,\th,\Si)$ where $\Xi$ is a topological groupoid, $\Si$ is a topological space and the maps $\rho,\th$ are continuous. Right actions are defined by analogy.
\end{defn}

If the action $\th$ is understood, we will write $\xi\bu\si$ instead of $\xi\bu_\th\si$. 
For $\xi\in\Xi\,,\,{\sf A},{\sf B}\subset\Xi\,,\,M\subset\Si$ we use the notations
\begin{equation*}\label{botations}
{\sf A}{\sf B}:=\big\{\xi\eta\,\big\vert\,\xi\in{\sf A}\,,\,\eta\in{\sf B}\,,\,\d(\xi)=\r(\eta)\big\}\,,
\end{equation*}
\begin{equation*}\label{otattions}
{\sf A}\bu M:=\big\{\xi\bu\si\,\big\vert\,\xi\in{\sf A}\,,\si\in M\,, \d(\xi)=\rho(\si)\big\}=\bigcup_{\xi\in{\sf A}}\xi\bu M.
\end{equation*} 

\begin{rem}\label{caofi}
In \cite{Wi} it is shown that if the groupoid $\Xi$ is open then ${\sf A}\bu M\subset \Si$ is open whenever the sets ${\sf A}\subset\Xi$ and $M\subset \Si$ are open. For topological group actions  the product ${\sf A}M$ is open provided that only the subset $M$ is open. In addition, if ${\sf A},{\sf B}$ are subsets of the group, ${\sf A}{\sf B}$ is open whenever at least one of the subsets is. Examples in \cite{FM} show that for groupoids this is not true, and this will require some special care. Note that, even for open groupoids, the translation $\xi\bu M$ of an open subset $M$ of $\Si$ could not be open.
\end{rem}

\begin{defn}\label{pinera}
We are going to use {\it orbits} $\mathfrak O_\si\!:=\Xi_{\rho(\si)}\!\bu\si$ and {\it orbit closures} $\overline{\mathfrak O}_\si$\,.  {\it The orbit equivalence relation} will be denoted by $\sim$\,. A subset $M\subset \Si$ is called {\it invariant} if $\xi \bu M\subset M$, for every $\xi\in\Xi$\,. 
\end{defn}

\begin{defn}\label{giudat}
For every $M,N\subset\Si$ one defines {\it the recurrence set} 
\begin{equation*}\label{rrecur}
\widetilde\Xi_M^N=\{\xi\in\Xi\!\mid\!(\xi\bu M)\cap N\ne\emptyset\}\subset\Xi_{\rho(M)}^{\rho(N)}\,.
\end{equation*}
\end{defn}

Actually, when $\rho$ is also injective, one has $\,\widetilde\Xi_M^N=\Xi_{\rho(M)}^{\rho(N)}$\,. This applies, in particular, to Example \ref{startlet}.

\begin{ex}\label{valtoare}
The topological groupoid $\Xi$ also acts on itself, with $\Si:=\Xi$\,, $\rho:=\r$ and $\xi\bu\eta:=\xi\eta$\,. Then
$$
\widetilde\Xi_M^N=\big\{\xi\in\Xi\,\big\vert\, \Xi^{\d(\xi)}\!\cap M\cap\xi^{-1}N\ne\emptyset\big\}\,.
$$
\end{ex}

\begin{ex}\label{vanatoare}
To model a usual continuous group action $(\G,\th,\Si)$\,, one just notes that $\G$ is an obvious topological groupoid with trivial unit space $X=\{\e\}$\,, only formed of the unit of the group. The function $\rho:\Si\to\{\e\}$ is therefore constant, and $\G\!\Join\!\Si=\G\!\times\!\Si$\,. The recurrence sets coincides with the usual ones \cite{Au,dV,Gl}, intensively involved in "classical" dynamical systems.
\end{ex}

\begin{ex}\label{vasnatoare}
More generally, suppose that $\Xi$ is a group bundle. This means that $\d=\r$ and, as a consequence, $\Xi$ may be written as the disjoint union over $X$ of the isotropy groups $\G_x\equiv\Xi_x^x$\,. It is easy to see that the action $(\Xi,\rho,\bu,\Si)$ reduces to a family $\big\{(\G_x,\bu_x,\Si_x)\,\big\vert\,x\in X\big\}$ of group actions (supplied by a suitable continuity condition), where $\Si_x\!:=\rho^{-1}(x)$\,. For $M,N\subset\Si$ one gets immediately 
$$
\widetilde\Xi_M^N=\bigsqcup_{x\in X}\big(\G_x\big)_{M\cap\Si_x}^{N\cap\Si_x}\,.
$$
\end{ex}

Other computations of recurrence sets may be found in \cite{FM}.

\subsection{Dynamical properties of groupoid actions}\label{sucritor}

Using mainly \cite{FM}, to which we refer for extra information, we review some dynamical properties a groupoid action might possess. Most of them are natural generalizations of similar notions for continuous group actions. Some of them (but not all) have already appeared in the literature. They will be treated from the point of view of (various types of) morphisms in the next sections.

\smallskip
We list first various notions related to (topological) transitivity. For group actions, one has $({\rm TT}_3)\Leftrightarrow ({\rm TT}_2)\Leftrightarrow ({\rm TT}_1)\Leftrightarrow ({\rm RT})$\,. Maybe unexpectedly, for groupoid actions only some implications are true; this is connected with the fact that in general, the closure of an invariant set could not be invariant. This may also occur for orbits, so besides the orbit closure $\overline{\mathfrak O}_\si$ of a point $\si$ one also introduces $\mathfrak C_\si\supset\overline{\mathfrak O}_\si$\,, the smallest closed invariant set containing $\si$\,. The situation becomes tame if the source map of the groupoid is open.

\begin{defn}\label{pricinoasa}
Let $\Th:=(\Xi,\rho,\bu,\Si)$ be a continuous groupoid action. We consider the conditions
\begin{enumerate}
\item[(${\rm T})$] There is just one orbit in $\Si$ ({\it transitivity}).
\item[(${\rm PT})$] There exists at least a dense orbit in $\Si$ ({\it pointwise transitivity}).
\item[(${\rm WPT})$] There exists a point $\si\in\Si$ such that $\mathfrak C_\si=\Si$ ({\it weak pointwise transitivity}).
\item[(${\rm TT}_1)$]
$\Si$ is not the union of two proper invariant closed subsets (or, equivalently: two open non-void invariant subsets of $\,\Si$ have non-trivial intersection).
\item[(${\rm TT}_2)$]
Each non-empty open invariant subset of $\,\Si$ is dense.
\item[(${\rm TT}_3)$]
Each invariant subset of $\,\Si$ is ether dense, or nowhere dense ({\it topological transitivity}).
\item[(${\rm RT})$]
For every $\,U,V\subset\Si$ open and non-void,  $\widetilde\Xi_U^V\ne\emptyset$ holds ({\it recurrent transitivity}).
\end{enumerate}
\end{defn}

\begin{defn}\label{acata}
The groupoid dynamical system is {\it minimal} if all its orbits are dense (i.e.\;if there are no proper closed invariant sets).
\end{defn} 

Of course, transitivity $\,\Rightarrow$ minimality $\,\Rightarrow$ pointwise transitivity. The concept of minimality can be adapted in an obvious way to closed invariant subsets $M\subset\Si$\,.

\begin{rem}\label{prostie}
It is shown in \cite[Sect.\,4]{FM} that the following implications hold:
\begin{enumerate}
\item
Transitivity $\,\Rightarrow$ pointwise transitivity $\,\Rightarrow$ recurrent transitivity. 
\item
One has $({\rm TT}_3)\Rightarrow ({\rm RT})\Rightarrow ({\rm TT}_2)\Rightarrow ({\rm TT}_1)\Leftarrow({\rm WPT})\Leftarrow({\rm PT})$\,. 
\item
If the source map $\d$ of the groupoid $\Xi$ is open, the first four conditions in 2. are equivalent. None of these implications is an equivalence in general. Weak pointwise transitivity does not imply $({\rm TT}_2)$ without openness. Pointwise transitivity does not imply topological transitivity.
\end{enumerate}
\end{rem}

Let $\mathcal K(T)$ denote the family of compact subsets of the topological space $T$. 

\begin{defn}\label{pescuit}
\textit{The limit set of the point} $\si\in\Si$ is the closed subset of $\Si$
\begin{equation*}\label{nicetry}
\mathfrak L_\si^\th\equiv\mathfrak L_\si\!:=\!\bigcap_{{\sf K}\in\mathcal K(\Xi)}\!\overline{(\Xi\!\setminus\!{\sf K})\bu\si}=\!\bigcap_{{\sf k}\in\mathcal K(\Xi_{\rho(\si)})}\!\overline{\big(\Xi_{\rho(\si)}\!\setminus{\sf k}\big)\bu\si}\,.
\end{equation*}
When $\si\in\mathfrak L_\si$ holds, we say that $\si$ is \textit{a recurrent point} and write $\si\in\Si_{\rm rec}$. 
\end{defn}

The limit set is void if $\Xi_{\rho(\si)}$ is compact, but it can also be void in other situations. It is contained in the orbit closure  of $\si$ and it expresses asymptotic properties of the orbit. Under certain assumptions, it is invariant and behaves as an attractor. Proofs and other properties of the limit sets and of the recurrence points may be found in \cite[Subsect.\,5.1,\,5.2]{FM}.

\begin{rem}\label{precisely}
The limit set $\mathfrak L_\si$ is characterized in \cite[Subsect.\,5.1]{FM} as being composed precisely of the points $\tau\in\Si$ such that for every neighborhood $V$ of $\tau$, the recurrence set $\widetilde\Xi_\si^V$ is not relatively compact.
\end{rem}

\begin{defn}\label{wandering}
The point $\si\in\Si$ is called {\it wandering} if $\widetilde{\Xi}_W^W$ is relatively compact for some neighborhood $W$ of $\si$\,. In the opposite case, it is called {\it non-wandering} (write $\si\in\Si_{\rm nw}$).
\end{defn}

The non-wandering points form a closed subset, containing all the limit points. These and other facts are shown in \cite[Subsect.\,5.2]{FM}.

\smallskip
Let us finish this section by indicating two other definitions taken from \cite{FM}.

\begin{defn}\label{intaresc}
{\it A fixed point} is a point $\si\in\Si$ such that $\xi\bu\si=\si$ for every $\xi\in\Xi_{\rho(\si)}$\,. This is equivalent with $\widetilde\Xi_\si^\si=\Xi_{\rho(\si)}$\,. We write $\si\in\Si_{\rm fix}$\,.
\end{defn}

\begin{defn}\label{indrasniesc}
\begin{enumerate}
\item[(a)] Let $x\in X$. The subset ${\sf A}$ of $\,\Xi_x$ is called {\it syndetic} if ${\sf KA}=\Xi_x$ for a compact subset ${\sf K}$ of $\,\Xi$\,.
\item[(b)] We say that $\si\in\Si$ is {\it periodic}, and we write $\si\in\Si_{\rm per}$\,, if $\,\widetilde\Xi_\si^\si$ is syndetic in $\Xi_{\rho(\si)}$\,.
\item[(c)] The point $\si$ is called \textit{weakly periodic} (we write $\si\in\Si_{\rm wper}$) if the subgroup $\widetilde\Xi_\si^\si$ is not compact.
\item[(d)] The point $\si\in\Si$ is said to be {\it almost periodic} if $\,\widetilde\Xi_\si^U$ is syndetic in $\Xi_{\rho(\si)}$ for every neighborhood $U$ of $\si$ in $\Si$\,.
We denote by $\Si_{\rm alper}$ the set of all the almost periodic points. 
\end{enumerate}
\end{defn}

Properties of the sets introduced in this definition are studied in \cite[Sect.\,6]{FM}. In \cite[Prop.\,6.12]{FM}, under the assumption that the $\d$-fibres of the groupoid $\Xi$ are not compact, some inclusion relations between the sets $\Si_{\rm fix},\Si_{\rm per},\Si_{\rm alper},\Si_{\rm rec},\Si_{\rm nw}$ are established. In the next sections we are going to need the following results, subject of \cite[Prop.\,6.16\,,\,6.18\,,\,Th.\,2.23]{FM}:
\begin{thm}\label{flacara}
Let $\Th:=(\Xi,\rho,\bu,\Si)$ be a continuous groupoid action.
\begin{enumerate}
\item[(i)]
The periodic points have compact orbits. If $\,\Xi$ is locally compact, second countable and $\d:\Xi\to\Xi$ is open, then the periodic points are exactly those having a compact orbit.
\item[(ii)]
Asuume that both $\Xi$ and $\Si$ are locally compact. If $\si$ is almost periodic, its orbit closure $\overline{\mathfrak O}_\si$ is minimal and compact. If, in addition, $\d:\Xi\to\Xi$ is open, $\si\in\Si$ is almost periodic if and only if its orbit closure $\overline{\mathfrak O}_\si$ is minimal and compact.
\end{enumerate}
\end{thm}

\subsection{Morphisms of groupoid actions}\label{goam}

\begin{defn}\label{sprevest}
Let $(\Xi,\rho,\th,\Si)\,,(\Xi',\rho',\th',\Si')$ be two groupoid dynamical systems. {\it A morphism of groupoid actions} is a pair $(\Psi,f)$\,, where $\Psi:\Xi\to\Xi'$ a continuous groupoid morphism (or continuous functor, if one prefers the category point of view) and $f:\Si\to\Si'$ is a continuous function, such that the next diagram commutes
\begin{equation*}\label{coliflor}
\begin{diagram}
\node{\Xi}\arrow{s,l}{\Psi}\arrow{e,b}{\d}\arrow{e,t}{\r}\node{X}\arrow{s,l}{\psi}\node{\Si}\arrow{w,t}{\rho}\arrow{s,l}{f}\\ 
\node{\Xi'}\arrow{e,t}{\r'}\arrow{e,b}{\d'}\node{X'}\node{\Si'}\arrow{w,t}{\rho'}
\end{diagram}
\end{equation*}
and such that
\begin{equation}\label{ciudata}
f\big[\th_{\xi}(\si)\big]=\th'_{\Psi(\xi)}\big[f(\si)\big]\,,\quad{\rm if}\ \ \d(\xi)=\rho(\si)\ (\textup{which implies}\ \d'[\Psi(\xi)]=\rho'\big[f(\si)\big])\,.
\end{equation}
The object map $\psi$ is the restriction of the arrow map $\Psi$. If both $\Psi$ and $f$ are surjective, we say that $(\Psi,f)$ is {\it an epimorphism}, $(\Xi',\rho',\th',\Si')$ is {\it a factor} of $(\Xi,\rho,\th,\Si)$ and $(\Xi,\rho,\th,\Si)$ is {\it an extension} of $(\Xi',\rho',\th',\Si')$\,. 
\end{defn}

The relation \eqref{ciudata} may be written as 
\begin{equation*}\label{ciudoasa}
f\circ\th_{\xi}=\th'_{\Psi(\xi)}\!\circ f\,,\quad{\rm on}\ \ \Si_{\d(\xi)}:=\rho^{-1}\big(\{\d(\xi)\})\,,
\end{equation*}
or even as
\begin{equation}\label{ciudatta}
f(\xi\bu\si)=\Psi(\xi)\bu'\!f(\si)\,,\quad{\rm if}\ \ \d(\xi)=\rho(\si)\,.
\end{equation}

\begin{ex}\label{startlet}
Each topological groupoid $\Xi$ acts continuously in a canonical way on its unit space $X$, with $\Si=X$ and $\rho={\rm id}_X$. The element $\xi$ sends $\d(\xi)$ into $\r(\xi)$\,, which can also be written (with the special notation) as $\xi\ast x:=\xi x\xi^{-1}$ whenever $\d(\xi)=x$\,. In this case, if $M,N\subset\Si$\,, one has $\widetilde\Xi_M^N=\Xi_M^N$\,, using the notation \eqref{faneaka}.
For every continuous groupoid action $(\Xi,\rho,\theta, \Si)$\,, there exists a unique epimorphism onto the canonical action of $\Xi$\,, of the form $({\rm id}_\Xi,\rho)$, so we say that {\it the canonical action is terminal}.
In addition, there is a one-to-one correspondence between continuous (groupoid) morphisms $\Psi:\Xi\to\Xi'$ and morphisms $(\Psi,f)$ from the terminal action of $\,\Xi$ on the terminal action of $\,\Xi'$. (One must have $f=\psi=\Psi|_X$.)
\end{ex}

\begin{ex}\label{novotel}
The subset $\Delta$ of the topological groupoid $\Xi$ is called a {\it subgroupoid} if for every $(\xi,\eta)\in(\Delta\!\times\!\Delta)\cap\Xi^{(2)}$ one has $\xi\eta\in\Delta$\,and $\xi^{-1}\in\Delta$\,. This subgroupoid is {\it wide} if $\Delta^{(0)}=\Xi^{(0)}$. Let $(\Xi,\rho,\th,\Si)$ be a continuous groupoid action and $\Delta$ a wide subgroupoid of $\Xi$\,. The restricted action is defined by keeping the same anchor map $\rho$ and restricting the map $\th$ to
\begin{equation*}\label{prinparti}
\Delta\!\Join\!\Si=\{(\xi,\si)\in\Delta\!\times\!\Si\!\mid\!\d(\xi)=\rho(\si)\}=\{(\xi,\si)\in\Xi\!\Join\!\Si\!\mid\!\xi\in\Delta\}\,.
\end{equation*}
Every wide subgroupoid gives raise to a natural morphism $(\iota,{\rm id}_\Si)$ given by the inclusion $\iota:\Delta\to\Xi$\,. 
\end{ex}

\begin{ex}\label{homoconstruction}
Let $(\Xi',\rho',\th',\Si')$ be an action. Suppose that $\Xi$ is a topological groupoid with unit space $X$ and $\Psi:\Xi\to\Xi'$ is a continuous morphism such that the restriction $\psi:X\to X'$ is an homeomorphism. Then we can make $\Xi$ act on $\Si'$ by defining  
$$
\rho:=\psi^{-1}\circ\rho'\quad \textup{ and }\quad\th_{\xi}(\si'):=\th'_{\Psi(\xi)}(\si')\,.
$$ 
To justify the second definition, note that 
$$
\d'[\Psi(\xi)]=\rho'(\si')\,\Leftrightarrow\,\psi[d(\xi)]=\rho'(\si')\,\Leftrightarrow\,d(\xi)=\psi^{-1}\big[\rho'(\si)\big]=\rho(\si')\,.
$$
In this case, $\Psi$ gives birth to a morphism $(\Psi,\psi)$ between the actions $(\Xi,\rho,\th,\Si)$ and $(\Xi',\rho',\th',\Si')$\,. All the needed verifications are straightforward.
\end{ex}

If $(\Psi_1,f_1)$ is a morphism from $(\Xi,\rho,\th,\Si)$ to $(\Xi',\rho',\th',\Si')$ and $(\Psi_2,f_2)$ is a morphism from $(\Xi',\rho',\th',\Si')$ to $(\Xi'',\rho'',\th'',\Si'')$\,, then $(\Psi_2 ,f_2)\circ(\Psi_1 , f_1):=(\Psi_2\circ\Psi_1 ,f_2\circ f_1)$ is a morphism from $(\Xi,\rho,\th,\Si)$ to $(\Xi'',\rho'',\th'',\Si'')$. {\it Actually, with this morphisms, the groupoid actions form a category}. 

\smallskip
At this point it is useful to show the effect of morphisms on recurrence sets; this will be used in subsequent sections in a more sophisticated setting. The result is already a small extension of \cite[Lemma\,7.5]{FM}.

\begin{prop}\label{label}
 Let $(\Psi,f)$ be a morphism between the actions $(\Xi,\rho,\th,\Si)$ and $(\Xi',\rho',\th',\Si')$\,. If $M,N\subset\Si$ then 
 \begin{equation}\label{nebazam}
 \Psi\big(\widetilde\Xi_{M}^{N}\big)\subset\big(\widetilde{\Xi}'\big)_{f(M)}^{ f(N)}\,.
 \end{equation}
 If $\,\Psi$ is surjective and $\psi,f$ are injective, in \eqref{nebazam} one obtains the equality. 
 \end{prop}
 
 \begin{proof}
 Let $\xi\in \widetilde\Xi_{M}^{N}$ and $\si\in M$ such that $\d(\xi)=\rho(\si)$ and $\xi\bu\si\in N$. Then 
 $$
 \d'[\Psi(\xi)]=\psi[\d(\xi)]=\psi[\rho(\si)]=\rho'[f(\si)]
$$ 
and
 $$
 \Psi(\xi)\bu'f(\si)=f(\xi\bu\si)\in f(N)\,,
 $$ 
which shows \eqref{nebazam}. On the other hand, pick $\big(\xi',f(\si)\big)\in\Xi'\Join f(M)$ such that $\xi'\bu'f(\si)\in f(N)$\,. If $\Psi$ is surjective, this $\xi'\in\big(\widetilde{\Xi}'\big)_{f(M)}^{f(N)}$ can be written as $\xi'=\Psi(\xi)$\,. Then 
$$
\psi[\d(\xi)]=\d'[\Psi(\xi)]=\d'(\xi')=\rho'[f(\si)]=\psi[\rho(\si)]\,.
$$ 
By injectivity of $\psi$ one gets $\d(\xi)=\rho(\si)$\,, so 
$$
\xi'\bu'f(\si)=\Psi(\xi)\bu'f(\si)=f(\xi\bu\si)\in f(N)\,.
$$ 
By injectivity of $f$ this implies that $\xi\bu\si\in N$, implying in its turn that $\xi\in\widetilde\Xi_{M}^{N}$\,, so the equality in \eqref{nebazam} follows.
 \end{proof}
 
For this type of morphisms everything that we proved in \cite[Subsect.\,7.1]{FM} holds, with the appropriate modifications. But by Example \ref{furnal}, the general vague morphisms of actions from Subsection \ref{sucritor} extend considerably the present concept, so the results from Subsection \ref{suscritor} largely cover this.

\section{Generalized vague morphisms}\label{supritor}

\subsection{Generalized vague morphisms of groupoid actions}\label{sucritor2}

Let $A$ be a topological space, $\Xi$ a topological groupoid with unit space $X$ and $\pi:A\to X$ a continuous surjection. One defines {\it the pull-back groupoid} \cite{Mac,MZ}
\begin{equation*}\label{labar}
\Xi(\pi,A):=\{(a,\xi,b)\in A\!\times\!\Xi\!\times\!A\!\mid\!\pi(a)=\r(\xi)\,,\,\pi(b)=\d(\xi)\}\,,
\end{equation*}
which is a closed subspace of $A\!\times\!\Xi\!\times\!A$\,, with structural maps
\begin{equation*}\label{lacar}
{\rm D}(a,\xi,b):=b\,,\quad{\rm R}(a,\xi,b):=a\,,
\end{equation*}
\begin{equation*}\label{ladar}
(a,\xi,b)(b,\eta,c):=(a,\xi\eta,c)\,,\quad(a,\xi,b)^{-1}:=\big(b,\xi^{-1},a\big)\,.
\end{equation*}
The unit space is identified to $A$ through $\big(a,\pi(a),a\big)\to a$\,. It is easy to check that the restriction of the second projection
\begin{equation*}\label{lafar}
\pi_*\equiv\Pi:\Xi(\pi,A)\to\Xi\,,\quad \Pi(a,\xi,b):=\xi
\end{equation*}
is a groupoid morphism. 

\smallskip
The next result will be needed in Subsection \ref{suscritor}.

\begin{lem}\label{joser}
If $\,\pi:A\to X$ is proper, $\Pi:\Xi(\pi,A)\to\Xi$ is also a proper map.
\end{lem}

\begin{proof}
If ${\sf K}\subset\Xi$ is compact, then
$$
\begin{aligned}
\Pi^{-1}({\sf K})&=\big\{(a,\xi,b)\in A\!\times\!{\sf K}\!\times\!A\,\big\vert\,\pi(a)=\r(\xi)\,,\,\pi(b)=\d(\xi)\big\}\\
&\subset\pi^{-1}[\r({\sf K})]\!\times\!{\sf K}\!\times\!\pi^{-1}[\d({\sf K})]\,.
\end{aligned}
$$
The last set is compact in $A\!\times\!\Xi\!\times\!A$\,, since $\r,\d$ are continuous and $\pi$ is proper. This is enough to prove the statement.
\end{proof}

If $\Xi$ acts on a topological space $\Si$\,, this action induces an action of $\Xi(\pi,A)$ in a natural manner; this will pave the way towards generalized vague morphisms of actions:

\begin{prop}\label{laker}
Suppose that $\Th:=(\Xi,\rho,\bu,\Si)$ is a groupoid dynamical system and that $\pi:A\to X$ is a continuous surjection. Let us set 
$$
\Th(\pi,A):=\big(\Xi(\pi,A),{\rm pr}_2,\tilde\bu,\Si\!\Join\! A\big)\,,
$$ 
where
\begin{equation*}\label{lagar}
\Si\!\Join\!A:=\{(\si,a)\in\Si\!\times\!A\!\mid\rho(\si)=\pi(a)\}\overset{\rm pr_2}{\longrightarrow}A\,,
\end{equation*}
\begin{equation}\label{joagar}
(a,\xi,b)\tilde\bu(\si,b):=(\xi\bu\si,a)\,.
\end{equation}
Then $\Th(\pi,A)$ is a new groupoid action. Denoting by ${\rm pr}_1$ the restriction of the first projection to $\Si\!\Join\! A$\,, 
\begin{equation*}\label{lahar}
(\Pi,{\rm pr}_1):\big(\Xi(\pi,A),{\rm pr}_2,\tilde\bu,\Si\!\Join\! A\big)\to(\Xi,\rho,\bu,\Si)
\end{equation*}
defines an epimorphism of actions.
\end{prop}

\begin{proof}
Surjectivity of the restriction ${\rm pr}_2:\Si\!\Join\! A\to A$ is a consequence of the surjectivity of the first anchor map $\rho$\,. Its continuity is clear. If in \eqref{joagar} $(\si,b)\in\Si\!\Join\! A$ and $(a,\xi,b)\in\Xi(\pi,A)$\,, then $\d(\xi)=\pi(b)=\rho(\si)$ and so $\xi\bu\si$ makes sense. Since an identification is involved, we check the first axiom of the action:
$$
{\rm pr}_2(\si,a)\tilde\bu(\si,a)=a\tilde\bu(\si,a)\equiv(a,\pi(a),a)\tilde\bu(\si,a)=\big(\pi(a)\bu\si,a\big)=\big(\rho(\si)\bu\si,a\big)=(\si,a).
$$
The second axiom follows essentially from the computations
$$
\big[(a,\xi,b)(b,\eta,c)\big]\tilde\bu(\si,c)=(a,\xi\eta,c)\tilde\bu(\si,c)=\big((\xi\eta)\bu\si,a\big)\,,
$$
$$
(a,\xi,b)\tilde\bu\big[(b,\eta,c)\tilde\bu(\si,c)\big]=(a,\xi,b)\tilde\bu\big(\eta\bu\si,b\big)=\big(\xi\bu(\eta\bu\si),a\big)\,,
$$
and the fact that $\bu$ itself is an action. Checking the properties of $(\Pi,{\rm pr}_1)$ is straightforward.
\end{proof}

\begin{ex}\label{saex}
If one starts, as in Example \ref{startlet}, with the canonical action $\bu\equiv\ast$ of $\Xi$ on its unit space $\Si\equiv X$, then $X\!\Join\! A=\{(\pi(a),a)\!\mid a\in A\}$\,, interpreted as the graph of $\pi:A\to X$, may be identified with $A$ (through the second projection). The new action reads
$$
(a,\xi,b)\,\tilde\ast\,(\pi(b),b):=(\xi\ast\d(\xi),a)=(\r(\xi),a)=(\pi(a),a)\,,
$$
which can be written simply as
$$
(a,\xi,b)\,\tilde\ast\,b=a\,.
$$
\end{ex}

\begin{ex}\label{myex}
If the groupoid $\Xi\equiv\G$ is a topological group, then $\pi(a)=\e$ for every $a\in A$ and $\G(\pi,A)$ may be identified with $\G\!\times\!(A\!\times\!A)$\,, the product between the group and the obvious pair groupoid. If, in addition, the group $\G$ acts on the topological space $\Si$\,, then $\Si\!\Join\!A=\Si\!\times\! A$ and the action $\tilde\bu$ is the product between the group action on $\Si$ and the canonical action of the pair groupoid on $A$\,.
\end{ex}

\begin{prop}\label{inzbor}
Let $M,N\subset\Si$\,. Then, with respect to the action introduced in Proposition \ref{laker}
\begin{equation}\label{inghesuita}
\widetilde{\Xi(\pi,A)}_{{\rm pr}_1^{-1}(M)}^{{\rm pr}_1^{-1}(N)}=\big\{(a,\xi,b)\in\Xi(\pi,A)\,\big\vert\,\xi\in\widetilde\Xi_M^N\big\}=\Pi^{-1}\big(\widetilde\Xi_M^N\big)\,.
\end{equation}
\end{prop}

\begin{proof}
Let $(a,\xi,b)\in\Xi(\pi,A)$\,, i.e.
$$
a,b\in A\,, \xi\in\Xi\,,\ \,{\rm with}\,\ \r(\xi)=\pi(a)\,,\,\d(\xi)=\pi(b)\,.
$$
Setting $\mathbf M:={\rm pr}_1^{-1}(M)$ and $\mathbf N:={\rm pr}_1^{-1}(N)$ one has
$$
\begin{aligned}
(a,\xi,b)\in\widetilde{\Xi(\pi,A)}_\mathbf M^\mathbf N&\,\Leftrightarrow\,\exists\,\si\in\Si\,,\,\d(\xi)=\rho(\si)=\pi(b)\,,\,(\si,b)\in\mathbf M\,,(a,\xi,b)\tilde\bu(\si,b):=(\xi\bu\si,a)\in\mathbf N\\
&\,\Leftrightarrow\,\exists\,\si\in M,\,\d(\xi)=\rho(\si)=\pi(b)\,,\,\xi\bu\si\in N\\
&\,\Leftrightarrow\,\xi\in\widetilde\Xi_M^N\,,
\end{aligned}
$$
and the result follows.
\end{proof}

\begin{defn}\label{meyer}
Let $\Xi,\Xi'$ two topological groupoids with unit spaces $X,X'$, respectively. A {\it generalized vague morphism} $\big((A,\pi),(A',\pi'),(\Gamma,\gamma)\big)$ between them is composed of two continuous surjections $\pi:A\to X$\,, $\pi':A'\to X'$\,, a continuous map $\gamma:A\to A'$ and a morphism $\Gamma:\Xi(\pi,A)\to\Xi'(\pi',A')$ such that its restriction to the unit spaces coincides with $\gamma$\,.
\end{defn}

For the last requirement in the definition, we took into consideration the identification of $A$ with $\Xi(\pi,A)^{(0)}$ and of $A'$ with $\Xi'(\pi',A')^{(0)}$, respectively. If not, the condition would be
$$
\Gamma(a,\pi(a),a)=\big(\gamma(a),\pi'(\gamma(a)),\gamma(a))\big)\,,\quad\forall\,a\in A\,.
$$

\begin{rem}\label{pisica}
Using the identities ${\rm D}'\circ\Gamma=\gamma\circ{\rm D}$ and ${\rm R}'\circ\Gamma=\gamma\circ{\rm R}$\,, one sees immediately that 
$$
\Gamma(a,\xi,b)=\big(\gamma(a),\Gamma_2(a,\xi,b),\gamma(b)\big)\,,
$$ 
 where $\Gamma_2:\Xi(\pi,A)\to\Xi'$ has to satisfy
\begin{equation*}\label{spotify}
\Gamma_2(a,\xi\eta,c)=\Gamma_2(a,\xi,b)\Gamma_2(b,\eta,c)\quad{\rm if}\ \r(\xi)=\pi(a)\,,\,\d(\xi)=\pi(b)=\r(\eta)\,,\,\d(\eta)=\pi(c)\,,
\end{equation*}
so $\Gamma_2$ is just a morphism from $\Xi(\pi,A)$ to $\Xi'$.
\end{rem}

\begin{ex}\label{germchin}
In \cite[Subsect.\,3.7]{MZ} only appears the case in which $A=A'$, $\gamma={\rm id}_A$ and $\Gamma$ is an isomorphism; then $(A,\pi,\pi',\Gamma)$ is called {\it a vague isomorphism}. On the other hand, {\it the vague functors} of \cite[Subsect.\,3.6]{MZ} essentially correspond to the case $A'=X'$ and $\pi'={\rm id}_{X'}$\,; then $\Xi'(\pi',A')$ may be identified with $\Xi'$, as explained in  Example \ref{furnal}.
\end{ex}

We upgrade the previous definition to the level of groupoid actions.

\begin{defn}\label{memeyer}
Let $\Th:=(\Xi,\rho,\bu,\Si)$ and $\Th':=\big(\Xi',\rho',\bu',\Si'\big)$ two continuous groupoid actions. A {\it generalized vague morphism of actions} 
$$
\Upsilon:=\big((A,\pi),(A',\pi'),(\Gamma,\gamma),h\big)
$$ 
between them is composed of a generalized vague morphism $\big((A,\pi),(A',\pi'),(\Gamma,\gamma)\big)$ from $\Xi$ to $\Xi'$ and a continuous function $h:\Si\to\Si'$, such that
\begin{enumerate}
\item[(i)] $\rho(\si)=\pi(a)\Leftrightarrow\rho'[h(\si)]=\pi'[\gamma(a)]$\,,
\item[(ii)] $(\Gamma,h\times\gamma)$ is a (usual) morphism from $\Th(\pi,A)$ to $\Th'(\pi',A)$\,.
\end{enumerate}
\end{defn}

\begin{rem}\label{marunt}
By (i), the restriction $h\times\gamma:\Si\!\Join\! A\to\Si'\!\Join\! A'$ is well-defined. Condition (i) will also be useful in the proof of Theorem \ref{both}. The condition (ii) means in detail that
\begin{equation}\label{rids}
\big(h(\xi\bu\si),\gamma(a)\big)=\Gamma(a,\xi,b)\tilde\bu'\big(h(\si),\gamma(b)\big)\quad{\rm if}\ \ \pi(a)=\d(\xi)=\rho(\si)\,,\;\pi(b)=\r(\xi)\,.
\end{equation}
Let us decompose $\Gamma(a,\xi,b)=:\big(\gamma(a),\Gamma_2(a,\xi,b),\gamma(b)\big)$\,, as in Remark \ref{pisica}. Then \eqref{rids} is satisfied if and only if, under the stated conditions on $a,b,\xi,\si$\,, one has 
\begin{equation}\label{HOM}
    h(\xi\bu\si)=\Gamma_2(a,\xi,b)\bu' h(\si)\,.
\end{equation}
\end{rem}

The next diagram might help:

\begin{equation*}\label{largedigram}
\begin{diagram}
\node{\Si}\arrow{s,l}{h}\arrow{e,t}{\rho}\node{X}\node{\Xi}\arrow{w,t}{\d}\arrow{w,b}{\r}\node{\Xi(\pi,A)}\arrow{w,t}{\Pi}\arrow{s,l}{\Gamma}\arrow{e,t}{{\rm D}}\arrow{e,b}{{\rm R}}\node{A}\arrow{s,l}{\gamma}\node{\Si\!\Join\!A}\arrow{s,r}{h\times\gamma}\arrow{w,t}{{\rm pr}_2}\\ 
\node{\Si'}\arrow{e,t}{\rho'}\node{X'}\node{\Xi'}\arrow{w,t}{\d'}\arrow{w,b}{\r'}\node{\Xi'(\pi',A')}\arrow{e,t}{{\rm D'}}\arrow{e,b}{{\rm R'}}\arrow{w,t}{\Pi'}\node{A'}\node{\Si'\!\Join\!A'}\arrow{w,t}{{\rm pr}'_2} 
\end{diagram}
\end{equation*}

Note that there are no maps $\Xi\to\Xi'$ or $X\to X'$, but at least $\Gamma_2:=\Pi'\circ\Gamma$ is a groupoid morphism.

\begin{ex}\label{furnal}
Let us explain how one obtains usual morphisms of groupoid actions as particular cases of generalized vague morphisms of groupoid actions. Let $\Th:=(\Xi,\rho,\bu,\Si)$ a groupoid dynamical system. Take $A:=X=\Xi^{(0)}$ and set $\pi:={\rm id}_X$. Then the pull-back groupoid will be isomorphic with $\Xi$ itself
$$
\Xi\big({\rm id}_X,X\big)\ni\big(\r(\xi),\xi,\d(\xi)\big))\overset{\Pi}{\longrightarrow}\xi\in\Xi
$$
and the space
$$
\Si\!\Join\!X=\big\{(\si,x)\,\big\vert\,\rho(\si)={\rm id}_X(x)\big\}=\big\{(\si,\rho(\si))\,\big\vert\,\si\in\Si)\big\}\,,
$$
seen as the graph of $\rho$\,, identifies with $\Si$ through the first projection. Making use of these identifications, the action $\tilde\bu$ reproduces the initial action $\bu$\,. Suppose now that $\Th:=(\Xi,\rho,\bu,\Si)$ and $\Th':=\big(\Xi',\rho',\bu',\Si'\big)$ are two continuous groupoid actions and that we take $A=X,\,\pi={\rm id}_X,\ A'=X',\,\pi'={\rm id}_{X'}$. Then a generalized vague morphisms of actions simply collapses to a usual morphism of groupoid actions $(\Gamma,h)$\,, as in Definition \ref{sprevest}. Implementing the identifications, \eqref{formula} below reduces to \eqref{nebazam}, with modified notations.
\end{ex}

\subsection{Dynamical properties under generalized vague morphisms of actions}\label{suscritor}

In this subsection, $\Upsilon$ will be a generalized vague morphism between the groupoid actions $\Th$ and $\Th'$, with detailed notations as in Definition \ref{memeyer}. One writes $\Upsilon:\Th\rightarrowtail\Th'$\,. We denote by $\mathfrak O_\si$ the orbit of $\si\in\Si$ with respect to $\Th$ and by $\mathfrak O'_{\si'}$ the orbit of $\si'\!\in\Si'$ with respect to $\Th'$. Same type of notations for the limit sets $\mathfrak L_{\si}$ and $\mathfrak L'_{\si'}$\,.

\begin{thm}\label{both}
Let $M,N\subset\Si$\,. Then
\begin{equation}\label{formula}
\Gamma\Big[\Pi^{-1}\big(\widetilde\Xi_M^N\big)\Big]\subset\big(\Pi'\big)^{-1}\Big[\big(\widetilde\Xi'\big)_{h(M)}^{h(N)}\Big]\,.
\end{equation}
If $\,\Gamma$ is surjective and $h,\gamma$ are injective, one has equality.
\end{thm}

\begin{proof}
By \eqref{inghesuita} we have
\begin{equation}\label{iesuita}
\widetilde{\Xi(\pi,A)}_{{\rm pr}_1^{-1}(M)}^{{\rm pr}_1^{-1}(N)}=\Pi^{-1}\big(\widetilde\Xi_M^N\big)
\end{equation}
and
\begin{equation}\label{iezuita}
\widetilde{\Xi'(\pi',A')}_{({\rm pr}'_1)^{-1}[h(M)]}^{({\rm pr}'_1)^{-1}[h(N)]}=(\Pi')^{-1}\Big[\big(\widetilde\Xi'\big)_{h(M)}^{h(N)}\Big]\,.
\end{equation}
Applying now Proposition \ref{label} to the morphism $\big(\Gamma,h\times\gamma\big)$\,, with ${\rm pr}_1^{-1}(M),{\rm pr}_1^{-1}(N)\subset\Si\Join A$\,, one gets
\begin{equation}\label{slaves}
\Gamma\Big[\widetilde{\Xi(\pi,A)}_{{\rm pr}_1^{-1}(M)}^{{\rm pr}_1^{-1}(N)}\Big]\subset\widetilde{\Xi'(\pi',A)}_{(h\times\gamma)[{\rm pr}_1^{-1}(M)]}^{(h\times\gamma)[{\rm pr}_1^{-1}(N)]}\,.
\end{equation}
But, using condition (ii), we have
\begin{equation}\label{fiesurjectiva}
(h\times\gamma)[{\rm pr}_1^{-1}(M)]\subset({\rm pr}_1')^{-1}[h(M)]\,,
\end{equation}
and similarly with $M$ replaced by $N$. Thus \eqref{slaves} and the obvious monotony of the recurrence sets imply
\begin{equation}\label{slawes}
\Gamma\Big[\widetilde{\Xi(\pi,A)}_{{\rm pr}_1^{-1}(M)}^{{\rm pr}_1^{-1}(N)}\Big]\subset\widetilde{\Xi'(\pi',A')}_{({\rm pr}_1')^{-1}[h(M)]}^{({\rm pr}_1')^{-1}[h(N)]}\,.
\end{equation}
All these result in
$$
\begin{aligned}
\Gamma\Big[\Pi^{-1}\big(\widetilde\Xi_M^N\big)\Big]&\overset{\eqref{iesuita}}{=}\Gamma\big[\widetilde{\Xi(\pi,A)}_{{\rm pr}_1^{-1}(M)}^{{\rm pr}_1^{-1}(N)}\big]\\
&\overset{\eqref{slawes}}{\subset}\widetilde{\Xi'(\pi',A')}_{({\rm pr}_1')^{-1}[h(M)]}^{({\rm pr}_1')^{-1}[h(N]}\\
&\overset{\eqref{iezuita}}{=}(\Pi')^{-1}\Big[\big(\widetilde\Xi'\big)_{h(M)}^{h(N)}\Big]\,.
\end{aligned}
$$
If $\Gamma$ is surjective, so is $\gamma:A\to A'$ and one easily gets an equality in \eqref{fiesurjectiva} (and similar for $N$).
If, in addition, both $h$ and $\gamma$ are injective, then $h\times\gamma$ is injective; by Proposition \ref{label} one will have equality in \eqref{slaves} and thus also in \eqref{slawes} and \eqref{formula}. This finishes the proof.
\end{proof}

\begin{rem}\label{stift}
From \eqref{formula} one deduces
\begin{equation}\label{vormula}
\Pi'\Big\{\Gamma\Big[\Pi^{-1}\big(\widetilde\Xi_M^N\big)\Big]\Big\}\subset\big(\widetilde\Xi'\big)_{h(M)}^{h(N)}
\end{equation}
and
\begin{equation}\label{bormula}
\widetilde\Xi_M^N\subset\Pi\Big(\Gamma^{-1}\Big\{\big(\Pi'\big)^{-1}\Big[\big(\widetilde\Xi'\big)_{h(M)}^{h(N)}\Big]\Big\}\Big)\,.
\end{equation}
Both inclusions become equalities whenever, in addition, $\,\Gamma$ is surjective and $h,\gamma$ are injective. 
\end{rem}

\begin{cor}\label{securinta}
Suppose that $\Upsilon:\Th\rightarrowtail\Th'$ is a generalized vague morphism of actions and that the map $h$ is surjective. If $\,\Th$ is recurrently transitive\,, then $\Th'$ is recurrently transitive.
\end{cor}

\begin{proof}
\smallskip
Let $\emptyset\ne U',V'\!\subset\Si'$ two open sets. Then, $h$ being surjective and continuous, $h^{-1}(U'),h^{-1}(V')$ are non-void open subsets of $\Si$\,, and $\Th$ being recurrently transitive, $\widetilde\Xi_{h^{-1}(U')}^{h^{-1}(V')}\ne\emptyset$\,. By surjectivity one has $h\big[h^{-1}(U')\big]=U'$ and $h\big[h^{-1}(V')\big]=V'$. Then one applies \eqref{formula} with $M=h^{-1}(U')$ and $N=h^{-1}(V')$ to show that $\big(\widetilde\Xi'\big)_{U'}^{V'}\ne\emptyset$\,; recall that $\Pi$ is surjective.
\end{proof}

\begin{thm}\label{color}
\begin{enumerate}
\item[(i)]If $\si\in\Si$\,, one has
\begin{equation}\label{potroc}
h\big(\mathfrak O_\si\big)\subset\mathfrak O_{h(\si)}'\quad{\rm and}\quad h\big(\overline{\mathfrak O}_\si\big)\subset\overline{\mathfrak O}'_{h(\si)}\,.
\end{equation}
\item[(ii)] If in addition $\Gamma$ is surjective and $h,\gamma$ are injective, then $h(\mathfrak O_\si)=\mathfrak O'_{h(\si)}$\,. In consequence, images through $h$ of $\bu$-invariant sets are $\bu'$-invariant sets.
\item[(iii)] The inverse image through $h$ of an $\bu'$-invariant set is a $\bu$-invariant set.
\end{enumerate}
\end{thm}

\begin{proof}
(i) One could prove this starting from the definitions. 
But this is also a consequence of Theorem \ref{both}, since $\si,\tau$ belong to the same orbit if and only the recurrence set corresponding to the sets $M=\{\si\}$ and $N=\{\tau\}$ is non-void. So assume that $\tau\in\mathfrak O_\si$\,, i.e.\;$\,\widetilde\Xi_\si^\tau\ne\emptyset$\,; since $\Pi$ is surjective, one gets $\Gamma\Big[\Pi^{-1}\big(\widetilde\Xi_M^N\big)\Big]\ne\emptyset$\,. By \eqref{formula} this implies $\big(\widetilde\Xi'\big)_{h(\si)}^{h(\tau)}\ne\emptyset$\,, meaning that $h(\si)$ and $h(\tau)$ belong to the same orbit of the action $\Th'$. This proves the first inclusion in  \eqref{potroc}. The second one is an immediate consequence.

\smallskip
(ii) Assuming that $\Gamma$ is surjective and $h,\gamma$ are injective, one has equality in \eqref{formula}. In particular, the two members are simultaneously non-void. From this and the fact that $\Pi,\Pi'$ are surjective, it follows that $\widetilde\Xi_\si^\tau$ and $\big(\widetilde\Xi'\big)_{h(\si)}^{h(\tau)}$ are simultaneously non-void, so  $\si\overset{\bu}{\sim}\tau$ if and only if $h(\si)\overset{\bu'}{\sim}h(\tau)$ and thus orbits are send (bijectively) onto orbits. Now let $M\subset\Si$ be invariant. To show that $h(M)$ is invariant, it is enough to check that for every $\si\in M$ the orbit $\mathfrak O'_{h(\si)}\!=h(\mathfrak O_\si)$ is contained in $h(M)$\,, which is obvious since $\mathfrak O_\si\subset M$.

\smallskip
(iii) Let $B'\subset\Si'$ be $\bu'$-invariant and let $\si\in\Si$ with $h(\si)\in B'$. Pick $\tau=\xi\bu\si$ an element on the orbit of $\si$. Then, by \eqref{HOM}, for some $a\in\pi^{-1}[\r(\xi)]$ and $b\in\pi^{-1}[\d(\xi)]$
$$
h(\tau)=h(\xi\bu\si)=\Gamma_2(a,\xi,b)\bu'h(\si)\in B',
$$
proving that $\tau\in h^{-1}(B')$\,.
\end{proof}

\smallskip
In the next consequence we mention mostly properties which depend directly on the invariant subset structure and maybe also on the topology.

\begin{cor}\label{secinta}
Let $\Upsilon:\Th\rightarrowtail\Th'$ be a generalized vague morphism of actions with $h$ surjective. 
\begin{enumerate}
\item[(i)] If $\,\Th$ is transitive, $\Th'$ is also transitive. If $\,\Th$ is pointwise transitive, $\Th'$ is also pointwise transitive.
\item[(ii)] If $\,\Th$ is weakly pointwise transitive, $\Th'$ is also weakly pointwise transitive.
\item[(iii)] If $\,\Th$ satisfies any of the properties $({\rm TT}_j)$\, with $j\in\{1,2\}$\,, then $\Th'$ also satisfies the same property. 
\item[(iv)] If $\,M\subset\Si$ is minimal and $h(M)\subset\Si'$ is closed, $h(M)$ is also minimal.
\end{enumerate}
\end{cor}

\begin{proof}
(i) is obvious from \eqref{potroc}.

\smallskip
(ii) Let $\si\in\Si$ such that $\mathfrak C_\si=\Si$\,. If $C'$ is a closed invariant set containing $h(\si)\in\Si'$, then $h^{-1}(C')$ is a closed invariant set containing $\si\in\Si$\,. Since $\mathfrak C_\si=\Si$ one has $h^{-1}(C')=\Si$\,; also using the surjectivity of $h$ one gets 
$$
C'=h\big[h^{-1}(C')]=h(\Si)=\Si',
$$
which means that $\Th'$ is weakly pointwise transitive.

\smallskip
(iii) We only treat the case $j=2$\,; the case $j=1$ is left to the reader. So pick a non-void open $\bu'$-invariant subset $U'$ of $\Si'$\,. The subset $h^{-1}(U')$ of $\Si$ is non-void, open and $\bu$-invariant, satisfying $h\big[h^{-1}(U')\big]=U'$ by surjectivity. Since $\Th$ satisfies $({\rm TT}_2)$ $h^{-1}(U')$ is dense, so $U'$ will also be dense. 


\smallskip
(iv) If $M$ is minimal and $\si\in M$ then
$$
h(M)=h\big(\overline{\mathfrak O}_\si\big)\subset\overline{\mathfrak O}'_{h(\si)}\subset\overline{h(M)}=h(M)\,,
$$
so the orbit of $h(\si)$ is dense in $h(M)$\,.
\end{proof}

\begin{rem}\label{cudat}
In \cite[Ex.\,7.9]{FM} it is shown that generally the property $({\rm TT}_3)$ does not transfer to factors, in the simple case of (usual) morphisms of groupoid actions. On the other hand, if the source map of the groupoid is open, the property does transfer even under generalized vague morphisms of actions, by the equivalence mentioned in Remark \ref{prostie}. 
\end{rem}

We present now a result on the lifting of minimality.

\begin{cor}\label{garbanzos}
Let $\Upsilon:\Th\rightarrowtail\Th'$ be a generalized vague morphism of actions with $h$ surjective.
Suppose that $\Si$ is compact (hence $\Si'$ is also compact). If $M'\subset\Si'$ is minimal, there exists $M\subset\Si$ minimal such that $h(M)=M'$. 
\end{cor}

\begin{proof}
The set $h^{-1}(M')$ is non-void closed and $\bu$-invariant. By Zorn's Lemma, it contains a minimal subset $M$. Then $h(M)\subset M'$ is non-void closed and $\bu'$-invariant, so it coincides with $M'$.  
\end{proof}

\begin{cor}\label{gogonata}
Let $\Upsilon:\Th\rightarrowtail\Th'$ be a generalized vague morphism of actions with $\Gamma$ surjective and $h,\gamma$ injective. Also assume that $\Xi'$ is locally compact, second countable, with $\d'\!:\Xi'\to\Xi'$ an open map. If $\si\in\Si$ is periodic, then $h(\si)\in\Si'$ is periodic.
\end{cor}

\begin{proof}
We use Theorem \ref{color}\,(ii) and Theorem \ref{flacara}\,(i). If $\si$ is periodic, $\mathfrak O_\si$ is compact, so $h(\mathfrak O_\si)=\mathfrak O'_{h(\si)}$ is compact, thus $h(\si)$ is periodic.
\end{proof}

\begin{cor}\label{rolar}
Let $\Upsilon:\Th\rightarrowtail\Th'$ be a generalized vague morphism of actions, where $\Xi\,,\Si\,,\Xi',\Si'$ are locally compact, $\d':\Xi'\to\Xi'$ is open, $\Gamma$ is surjective and $\gamma,h$ are injective. If $\si$ is almost periodic, $h(\si)$ is almost periodic.
\end{cor}

\begin{proof}
We are going to use Theorem \ref{flacara}\,(ii), Corollary \ref{secinta}\,(iv) and Theorem \ref{color}\,(ii). Since $\si$ is almost periodic $\overline{\mathfrak O}_\si$ is minimal and compact, so $h\big(\overline{\mathfrak O}_\si\big)$ is also minimal and compact. This set is equal to $\overline{\mathfrak O}'_{h(\si)}$\,, so $h(\si)$ is almost periodic.
\end{proof}

We prove next a different result about the propagation of periodicity and almost periodicity under generalized vague morphism of actions.  It should be compared with Corollaries \ref{gogonata} and \ref{rolar}.

\begin{prop}\label{rollar}
Let $\Upsilon:\Th\rightarrowtail\Th'$ be a generalized vague morphism of actions with $\Gamma$ surjective, $\pi$ injective and proper and $\pi'\circ\gamma$ injective. 
\begin{enumerate}
\item[(i)]
If $\si$ is almost periodic, $h(\si)$ is almost periodic.
\item[(ii)]
If $\si$ is periodic, $h(\si)$ is periodic.
\end{enumerate}
\end{prop}

\begin{proof}
Let us start with some preparations. Note that $\Pi'\circ\Gamma=\Gamma_2$\,. It is a groupoid epimorphism $\Xi(\pi,A)\to\Xi'$ and its restriction to units $\Xi(\pi,A)^{(0)}\equiv A\to X'$ is exactly $\pi'\circ\gamma$\,, supposed injective here. It follows immediately that
\begin{equation}\label{sprejos}
\Gamma_2\big[\Xi(\pi,A)_b\big]=\Xi'_{\pi'[\gamma(b)]}\,,\quad\forall\,b\in A\,.
\end{equation}
This result also follows from Proposition \ref{label} applied to the canonical actions of the two groupoids. On the other hand, remark that the following diagram commutes:
\begin{equation*}\label{largedigoram}
\begin{diagram}
\node{\Si}\arrow{s,l}{h}\arrow{e,t}{\rho}\node{X}\node{A}\arrow{s,r}{\gamma}\arrow{w,t}{\pi}\\ 
\node{\Si'}\arrow{e,t}{\rho'}\node{X'}\node{A'}\arrow{w,t}{\pi'}
\end{diagram}
\end{equation*}
Since $\pi$ is injective, this translates into
\begin{equation}\label{translates}
\pi'\circ\gamma\circ\pi^{-1}\!\circ\rho=\rho'\circ h\,.
\end{equation}

\smallskip
(i) Let $V$ be an open neighborhood of $h(\si)\in\Si'$. Then $U\!:=h^{-1}(V)$ is an open neighborhood of $\si\in\Si$\,, so there exists a compact subset ${\sf K}\subset\Xi$ such that ${\sf K}\,\widetilde\Xi_\si^U=\Xi_{\rho(\si)}$\,. Define ${\sf K}'\!:=\Gamma_2\big[\Pi^{-1}({\sf K})\big]$\,, which is compact, by Lemma \ref{joser}. We can write
\begin{align*}
    {\sf K}'{\big(\widetilde\Xi'\big)}_{h(\si)}^{V}\supset{\sf K}'{\big(\widetilde\Xi'\big)}_{h(\si)}^{h(U)}&\overset{\eqref{vormula}}{\supset}\Gamma_2\big[\Pi^{-1}({\sf K})\big]\,\Gamma_2\Big[\Pi^{-1}\big(\widetilde\Xi_\si^U\big)\Big] \\
    &\supset\Gamma_2\Big[\Pi^{-1}({\sf K})\,\Pi^{-1}\big(\widetilde\Xi_\si^U\big)\Big] \\
    &\supset \Gamma_2\Big[\Pi^{-1}\big({\sf K}\,\widetilde\Xi_\si^U\big)\Big] \\
    &=\Gamma_2\Big[\Pi^{-1}\big(\Xi_{\rho(\si)}\big)\Big] \\
    &= \Gamma_2\Big[\Xi(\pi,A)_{\pi^{-1}(\rho(\si))}\Big] \\
    &\overset{\eqref{sprejos}}{=} \big(\Xi'\big)_{\pi'[\gamma(\pi^{-1}(\rho(\si)))]}\\
    &\overset{\eqref{translates}}{=}\big(\Xi'\big)_{\rho'(h(\si))}\,,
\end{align*} 
so $h(\si)\in\Si'$ is almost periodic. The third inclusion looks abstractly 
$$
\Pi^{-1}({\sf S})\Pi^{-1}({\sf T})\supset\Pi^{-1}({\sf ST})\,,\quad{\sf S,T}\subset\Xi
$$ 
and is left to the reader (note that $\Pi$ is injective, by the injectivity of $\pi$).

\smallskip
(ii) We assume now that $\si$ is periodic, i.\,e.\;that ${\sf L}\,\widetilde\Xi_\si^\si=\Xi_{\rho(\si)}$ for some compact subset ${\sf L}\subset\Xi$\,. Define ${\sf L}':=\Gamma_2\big(\Pi^{-1}({\sf L})\big)$\,, which is a compact set. One can repeat the computations above, replacing ${\sf K},{\sf K}',V$ by ${\sf L},{\sf L}',\{h(\si)\}$\,, respectively; instead of $U=h^{-1}(V)$ we take simply $\{\si\}$\,. One gets 
$$
{\sf L}'{\big(\widetilde\Xi'\big)}_{h(\si)}^{h(\si)}\supset\big(\Xi'\big)_{\rho'(h(\si))}\,,
$$ 
which shows that $h(\si)$ is periodic. 
\end{proof}


\begin{prop}\label{caciu}
Suppose that $\Upsilon$ is a generalized vague morphism between the groupoid actions $\Th$ and $\Th'$. 
If $\,\Gamma$ and $\pi'$ are proper maps, then for every point $\si\in\Si$ we have  $h\big(\mathfrak L_\si\big)\subset\mathfrak L_{h(\si)}'$\,. In consequence, $h(\si)$ is $\th'$-recurrent if $\si$ is $\th$-recurrent. In addition, $h(\si)$ is non-wandering if $\si$ is non-wandering.
\end{prop}

\begin{proof}
We are going to use the characterization of limit sets supplied in Remark \ref{precisely}. Let $\tau'=h(\tau)$ with $\tau\in\mathfrak L_\si$ and let $V'$ be a neighborhood of $\tau'$. Then $V:=h^{-1}(V')$ is a neighborhood of $\tau$ and $h(V)\subset V'$. Since $\big(\widetilde{\Xi}'\big)_{h(\si)}^{h(V)}\!\subset\!\big(\widetilde{\Xi}'\big)_{h(\si)}^{V'}$\,, it is enough to show that $\big(\widetilde{\Xi}'\big)_{h(\si)}^{h(V)}$ is not relatively compact. Since $\tau$ is a limit point, $\widetilde{\Xi}_\si^V$ is not relatively compact. Suppose however that $\big(\widetilde{\Xi}'\big)_{h(\si)}^{h(V)}$ is relatively compact. Formula \eqref{bormula} and the fact that $\Gamma$ and $\Pi'$ (see Lemma \ref{joser}) are proper and $\Pi$ is continuous show that $\widetilde\Xi_\si^V$ is also relatively compact and we finished the proof of the inclusion by contradiction.

\smallskip
The statement about recurrent points follows from this and from the definition.

\smallskip
To prove the last statement, one must show that $\si$ is wandering if $h(\si)$ is wandering. This goes as in the proof concerning limit sets, but with the pair $\big(\{\tau'\},V'\big)$ replaced by the pair $\big(W',W'\big)$\,, where $W'$ is a neighborhood of $h(\si)$ such that $\big(\widetilde{\Xi}'\big)_{W'}^{W'}$ is relatively compact.
\end{proof}

\section{Algebraic morphisms}\label{sucitorr}

\subsection{Algebraic morphisms of groupoid actions}\label{sucitor}

\begin{defn}\label{morfalgulus}
Let $\big(\Xi_1,\rho_1,\bu_1,\Si\big)$ be a left action and $\big(\Xi_2,\rho_2,\bu_2,\Si\big)$ a right action. We say that {\it the two actions commute} if 
\begin{enumerate}
\item[(i)] $\rho_1\big(\si\bu_2\xi_2\big)=\rho_1(\si)$ \;if \;$\r_2(\xi_2)=\rho_2(\si)$\,,
\item[(ii)] $\rho_2\big(\xi_1\bu_1\si\big)=\rho_2(\si)$ \;if \;$\d_1(\xi_1)=\rho_1(\si)$\,,
\item[(iii)] $\xi_1\bu_1\big(\si\bu_2\xi_2\big)=\big(\xi_1\bu_1\si\big)\bu_2\xi_2$ \;if \;$\d_1(\xi_1)=\rho_1(\si)$\; and\; $\r_2(\xi_2)=\rho_2(\si)$\,.
\end{enumerate}
\end{defn}

The following notion was taken from \cite{Bu,BS,BEM}:

\begin{defn}\label{morfalg}
Let $\Xi\,,\Xi'$ be two groupoids, with unit spaces $X$ and $X'$, respectively. {\it An algebraic morphism} (also called {\it an actor}) $\Xi\rightsquigarrow\Xi'$ is an action $\big(\Xi,\mu,\diamond,\Xi'\big)$ of $\Xi$ on the topological space $\Xi'$, which commutes with the right action of $\Xi'$ on itself by right multiplications.
\end{defn}

\begin{rem}\label{commut}
In this case 
$$
\Xi_1=\Xi\,,\quad\Xi_2=\Si=\Xi',\quad\rho_1=\mu:\Xi'\to X,\quad\bu_1=\diamond\,,\quad\rho_2=\d':\Xi'\to X'
$$ 
and commutativity means that
\begin{equation}\label{beans}
\mu(\eta'\xi')=\mu(\eta')\,,\quad{\rm if}\ \ \d'(\eta')=\r'(\xi')\,,
\end{equation}
\begin{equation}\label{teans}
\d'(\xi\diamond\eta')=\d'(\eta')\,,\quad{\rm if}\ \ \d(\xi)=\mu(\eta')\,,
\end{equation}
\begin{equation*}\label{means}
\xi\diamond(\eta'\xi')=(\xi\diamond\eta')\xi',\quad{\rm if}\ \ \d(\xi)=\mu(\eta')\,,\ \d'(\eta')=\r'(\xi')\,.
\end{equation*}
We are going to denote by $\nu$ the restriction of $\mu$ to $X'\!=\Xi'^{(0)}$. From \eqref{beans} one infers immediately that $\mu=\nu\circ\r'$\,; it follows that here $\mu$ is determined by $\nu$.
\end{rem}

\begin{ex}\label{proposit}
For two topological groups $\G,\G'$, algebraic morphisms are exactly topological group morphisms $\beta:\G\to\G'$. In this case $X=\{\e\}$ and $X'=\{\e'\}$\,, so $\mu(\xi')=\e$ for every $\xi'\in\G'$, and then $\xi\diamond\xi':=\beta(\xi)\xi'$ defines the action (note that $\beta(\xi)=\xi\diamond\e'$\,)\,.
\end{ex}

\begin{ex}\label{procopsit}
More generally (cf. Example \ref{vasnatoare}), let 
$$
\Xi=\bigsqcup_{x\in X}\G_x\,,\quad\Xi'=\bigsqcup_{x'\in X'}\G'_{x'}
$$ 
two group bundles and $(\Xi,\mu,\diamond,\Xi')$ an algebraic morphism. We recall that $\mu=\nu\circ\r'$ is determined by $\nu:X'\to X$. The elements $\xi\in\Xi_x$ and $\xi'\!\in\Xi'_{x'}$ may be composed by $\diamond$ if and only if $x=\nu(x')$\,. The action $\diamond$ reduces to a family of actions 
$$
\diamond_{x'}:\G_{\nu(x')}\!\times\G'_{x'}\to\G'_{x'}\,,\quad x'\in X'
$$
and finally, by Example \ref{proposit}, to a family of topological group morphisms
$$
\beta_{x'}:\G_{\nu(x')}\to\G'_{x'}\,,\quad x'\in X'
$$ tied together by some continuity condition.
\end{ex}

\begin{ex}\label{poposit}
If $\Xi'$ is a trivial groupoid (this is $\Xi'=X'$), the only possible algebraic morphism $\Xi\rightsquigarrow\Xi'$ is the one induced by the trivial action $\xi\diamond x'=x',\,\forall\,\xi\in\Xi_{\mu(x')}$\,(and this can only happen if $\d(\xi)=\r(\xi)$ for every $\xi\in\Xi_{\mu(X')}$). So it may be identified with the anchor map $\mu:\Xi'\to X$.
If, in addition, the groupoid $\Xi$ is also trivial, the algebraic morphism reduces to a continuous map $\Xi'=X'\!\overset{\mu}{\to}X=\Xi$ (note the inverse direction) and the outer multiplication is
\begin{equation*}\label{sviciovici}
\mu(x')\diamond x'\!:=x',\quad\forall\,x'\in X'.
\end{equation*}
\end{ex}

There is a way to multiply algebraic morphisms $\Xi_1\!\overset{\Phi_{12}}{\rightsquigarrow}\!\Xi_2\!\overset{\Phi_{23}}{\rightsquigarrow}\!\Xi_3$\,, resulting in the new algebraic morphism $\Xi_1\!\overset{\Phi_{13}}{\rightsquigarrow}\!\Xi_3$\,, where $\Phi_{13}\!:=\Phi_{23}\circ\Phi_{12}$\,: Let 
$\Phi_{12}:=\big(\Xi_1,\mu_{12},\diamond_{12},\Xi_2\big)$ and $\Phi_{23}:=\big(\Xi_2,\mu_{23},\diamond_{23},\Xi_3\big)$ be the two algebraic morphisms. We set
\begin{equation}\label{gombun}
\Phi_{13}:=\Phi_{23}\circ\Phi_{12}\equiv\big(\Xi_1,\mu_{13},\diamond_{13},\Xi_3\big)\,,
\end{equation}
\begin{equation}\label{pebucati}
\mu_{13}:=\mu_{12}\circ\mu_{23}:\Xi_3\to X_1\subset\Xi_1\,,
\end{equation}
\begin{equation}\label{prinscaieti}
\xi_1\diamond_{13}\xi_3:=\big(\xi_1\diamond_{12}\mu_{23}(\xi_3)\big)\diamond_{23\,}\xi_3\quad{\rm if}\quad\d_1(\xi_1)=\mu_{13}(\xi_3)=\mu_{12}\big[\mu_{23}(\xi_3)\big]\,.
\end{equation}
Both compositions in the r.h.s. of \eqref{prinscaieti} are possible; in particular, by \eqref{teans}, one has
$$
\d_2\big(\xi_1\diamond_{12}\mu_{23}(\xi_3)\big)=\d_2\big(\mu_{23}(\xi_3)\big)=\mu_{23}(\xi_3)\,.
$$ 

\begin{defn}\label{morfalg}
Let $\Th\!:=(\Xi,\rho,\th=\bu,\Si)$ and $\Th'\!:=(\Xi',\rho',\th'=\bu',\Si')$ be two groupoid actions, with unit spaces $X$ and $X'$, respectively. {\it An algebraic morphism of actions} $\Th\rightsquigarrow\Th'$ is a pair $(\Phi,g)$\,, where $\Phi\equiv(\Xi,\mu,\diamond,\Xi'):\Xi\rightsquigarrow\Xi'$ is an algebraic morphism, $g:\Si\to\Si'$ is a continuous function, such that 
\begin{equation}\label{fuame}
\rho=\nu\circ\rho'\circ g\,,
\end{equation}
\begin{equation}\label{siete}
g(\xi\bu\si)=\big[\xi\diamond\rho'(g(\si))\big]\bu'\!g(\si)\quad\textup{ whenever}\quad \d(\xi)=\rho(\si)=\nu\big[\rho'(g(\si))\big]\,.
\end{equation}
\end{defn}

The composability condition $\d'\big[\xi\diamond\rho'(g(\si))\big]=\rho'(g(\si))$ is automatic, by \eqref{teans}. Part of the story is told in the diagram
\begin{equation*}\label{copilofloc}
\begin{diagram}
\node{\Xi}\arrow{r,t}{\d}\node{X}\node{\Si}\arrow{w,t}{\rho}\arrow{s,r}{g}\\ 
\node{\Xi'}\arrow{ne,l}{\mu}\arrow{e,t}{\d'}\node{X'}\arrow{n,r}{\nu}\node{\Si'}\arrow{w,t}{\rho'}
\end{diagram}
\end{equation*}

\begin{ex}\label{terminalg}
Assume that both of the groupoids are groups: $\Xi=\G$\,, $\Xi'=\G'$, acting in the topological spaces $\Si,\Si'$, respectively. Recalling the setting of Example \ref{proposit} and the fact that $X=\{\e\}\,,X'=\{\e'\}$\,, condition \eqref{siete} translates into 
$$
g(\xi\bu\si)=\big[\xi\diamond\rho'(g(\si))\big]\bu'\!g(\si)=\big[\beta(\xi)\rho'(g(\si))\big]\bu'\!g(\si)=\beta(\xi)\bu'g(\si)\,,
$$ 
so algebraic morphisms become ordinary morphisms of (group) actions. (Note that the condition \eqref{fuame} gets trivialized)
\end{ex}

\begin{ex}\label{terminale}
If both action $\Th$ and $\Th'$ are canonical (terminal) actions, as in Example \ref{startlet}, then the above conditions translate into
\begin{equation*}\label{fuamme}
{\rm id}_X=\nu\circ g\,,
\end{equation*}
\begin{equation*}\label{siette}
\r'\big[\xi\diamond g(\d(\xi))\big]=g(\r(\xi))\,.
\end{equation*} Maybe it is interesting to note that $\d'\big[\xi\diamond g(\d(\xi))\big]=g(\d(\xi))$ also holds, because of condition \eqref{teans}.
\end{ex}

\begin{defn}\label{composition}
{\it The composition of two algebraic morphisms of actions} 
\begin{equation*}\label{trantit}
\big(\Xi_1,\rho_1,\bu_1,\Si_1\big)\overset{\Phi_{12}}{\underset{g_{12}}{\rightsquigarrow}}\big(\Xi_2,\rho_2,\bu_2,\Si_2\big)\underset{g_{23}}{\overset{\Phi_{23}}{\rightsquigarrow}}\big(\Xi_3,\rho_3,\bu_3,\Si_3\big)\,,
\end{equation*} 
where
\begin{equation*}\label{vraf}
\Phi_{12}\equiv\big(\Xi_1,\mu_{12},\diamond_{12},\Xi_2\big)\,,\quad\Phi_{23}\equiv\big(\Xi_2,\mu_{23},\diamond_{23},\Xi_3\big)
\end{equation*} 
is naturally defined as 
$$
(\Phi_{23},g_{23})\circ(\Phi_{12},g_{12}):=(\Phi_{23}\circ\Phi_{12},g_{23}\circ g_{12})\,,
$$ 
where $\Phi_{23}\circ\Phi_{12}$ comes from \eqref{gombun},\,\eqref{pebucati},\,\eqref{prinscaieti}  and $g_{23}\circ g_{12}$ is the usual composition of functions.
\end{defn}

\begin{prop}\label{saspermam}
The composition of two algebraic morphisms of actions is also an algebraic morphism of actions.
\end{prop}

\begin{proof}
We are going to use the notations of Definition \ref{composition}. At the level of anchor maps one has
\begin{equation*}\label{scursici}
\rho_1=\nu_{12}\circ \rho_2\circ  g_{12}=\nu_{12}\circ\big(\nu_{23}\circ \rho_3\circ g_{23}\big)\circ g_{12}=\nu_{13}\circ \rho_3\circ  g_{13}.
\end{equation*} We must also verify that, if $\d_1(\xi_1)=\rho_1(\si_1)$, then
\begin{equation*}\label{shapte}
g_{23}\big( g_{12}(\xi_1\bu_1\si_1)\big)=\big[\xi_1\diamond_{13}\rho_3(g_{23}( g_{12}(\si_1)))\big]\bu_3\!g_{23}\big( g_{12}(\si_1)\big)
\end{equation*} 
holds. Indeed, we will prove this, starting from the left hand side:
$$
\begin{aligned}
g_{23}\big( g_{12}(\xi_1\bu_1\si_1)\big)&\overset{\eqref{siete}}{=} g_{23}\big(\big[\xi_1\diamond_{12}\rho_2(g_{12}(\si_1))\big]\bu_2 g_{12}(\si_1)\big)\\
&\overset{\eqref{siete}}{=} \Big[\big[\xi_1\diamond_{12}\rho_2(g_{12}(\si_1))\big]\diamond_{23}\rho_3(g_{23}(g_{12}(\si_1)))\Big]\bu_{3}g_{23}\big( g_{12}(\si_1)\big) \\
&\overset{\eqref{fuame}}{=}
\Big[\big[\xi_1\diamond_{12}\nu_{23}\big(\rho_3(g_{23}(g_{12}(\si_1)))\big)\big]\diamond_{23}\rho_3\big(g_{23}(g_{12}(\si_1))\big)\Big]\bu_{3}g_{23}\big( g_{12}(\si_1)\big) \\
&\overset{\eqref{prinscaieti}}{=}\big[\xi_1\diamond_{13}\rho_3\big(g_{23}(g_{12}(\si_1))\big)\big]\bu_3 g_{23}\big( g_{12}(\si_1)\big)\,.
\end{aligned}
$$
\end{proof} 

It is not hard to see that the composition of algebraic morphisms is associative and has an identity morphism for every action: If $\Th=(\Xi,\rho,\bu,\Si)$ is an action, the algebraic morphism of actions consisting of the action of $\Xi$ on itself by left multiplication and the identity function $g={\rm id}_\Si$ is the identity morphism associated to $\Th$\,. So we are in presence of a category whose objects are groupoid actions and the arrows are algebraic morphisms.

\smallskip
Let us finish this subsection with a lemma which briefly explores some of the dynamics of the action $\diamond$\,; it will be used below. We denote by $\mathfrak O_{\xi'}^\diamond$ the orbit of $\xi'\!\in\Xi'$ under this action. ${\rm Sat}^\diamond(\cdot)$ represents the saturation (i.e. the smallest invariant set containing $\cdot$) with respect to the action $\diamond$\,.

\begin{lem}\label{constant}
The function $\d':\Xi'\to X'$ is constant on the orbits of $\diamond$\,. That is, if $\xi_1'\overset{\diamond}{\sim}\xi_2'$\,, then $\d'(\xi_1')=\d'(\xi_2')$\,. One may write this as
$$
\mathfrak O_{\xi'}^\diamond\subset \big(\Xi'\big)_{\d'(\xi')}\,,\quad\forall\,\xi'\in\Xi'.
$$ 
Moreover, if $\,{\rm Sat}^\diamond\big[\rho'(g(\Si))\big]=\Xi'$, the equality is achieved.
\end{lem}

\begin{proof}
If $\xi_1'$ and $\xi_2'$ are related by $\xi_1'=\xi\diamond\xi_2'$\,, then 
$$
\d'(\xi_1')=\d'\big(\xi\diamond\xi_2'\big)=\d'(\xi_2')\,.
$$ 
If now ${\rm Sat}^\diamond\big[\rho'(g(\Si))\big]=\Xi'$, for every pair of elements $\xi_1',\xi_2'\in\big(\Xi'\big)_{\d'(\xi')}$ we have 
$$
\r'\big((\xi_2')^{-1}\big)=\d'(\xi_2')=\d'(\xi_1')\,,
$$ 
so there exists some $\xi\in\Xi$ and some $x'=\rho'(g(\si))\in \rho'(g(\Si))$ such that 
$$
\xi\diamond x'=\xi_1'(\xi_2')^{-1}.
$$ 
Observe that 
$$
\r'(\xi_2')=\d'(\xi_1'(\xi_2')^{-1})=\d'(\xi\diamond x')=x',
$$ 
so in fact we have $\xi\diamond \r'(\xi_2')=\xi_1'(\xi_2')^{-1}$. This implies $\xi\diamond\xi_2'=\xi_1'$\,, so $\xi_1'$ and $\xi_2'$ belong to the same $\diamond$-orbit. 
\end{proof}

In general, we will use the assumption ${\rm Sat}^\diamond\big[\rho'(g(\Si))\big]=\Xi'$ as being the analogous of surjectivity. A justification for this is given in Remark \ref{image}. This assumption is stronger than ${\rm Sat}^\diamond(X')=\Xi'$, but weaker than ${\rm Sat}^\diamond(X')=\Xi$ and $g$ surjective. It will play an important role in the next subsection.

\begin{rem}
If ${\rm Sat}^\diamond\big[\rho'(g(\Si))\big]=\Xi'$, then for every $x'\in X'$, there exists $\si\in\Si$ and $\xi\in\Xi$ such that $\xi\diamond \rho'(g(\si))=x'$. This implies that 
$$
x'=\d'(x')=\d'\big(\xi\diamond \rho'(g(\si))\big)=\d'\big(\rho'(g(\si))\big)=\rho'(g(\si))\,,
$$ 
which means that $\rho'\circ g$ must be surjective.
\end{rem}

\subsection{Dynamical properties under algebraic morphisms of groupoid actions}\label{sucrier}

Let us fix  an algebraic morphism of groupoid actions 
$$
(\Phi,g):(\Xi,\rho,\bu,\Si)\rightsquigarrow\big(\Xi',\rho',\bu',\Si'\big)\,,\ \ {\rm with}\ \ \Phi=(\mu,\diamond)\,.
$$ 
Recall that the anchor maps $\rho,\rho'$ are assumed surjective. If necessary, we indicate the action by an upper label.

Let us start with the relationship between recurrence sets of the two actions. 

\begin{thm}\label{liema}
For $M,N\subset\Si$ we have
\begin{equation}\label{relashn}
\widetilde\Xi_{M}^N\diamond\rho'\big(g(M)\big)\subset\big(\widetilde\Xi'\big)_{g(M)}^{g(N)}\,.
\end{equation}
If $\,{\rm Sat}^\diamond\big[\rho'(g(\Si))\big]=\Xi'$ and $g$ is injective, equality is achieved. 
\end{thm}

\begin{proof}
Let $\xi\diamond\rho'\big(g(\si)\big)$ be an element of $\widetilde\Xi_{M}^N\diamond\rho'\big(g(M)\big)$\,, with $\xi\in \widetilde\Xi_{M}^N$ and $\si\in M$. This implies that 
$$
\d(\xi)=\nu\big[\rho'(g(\si))\big]\overset{\eqref{fuame}}{=}\rho(\si)\,.
$$ 
So one has
$$
\big[\xi\diamond\rho'(g(\si))\big]\bu'g(\si)=g(\xi\bu\si)\in g(N)\,.
$$ 
It follows that $\xi\diamond\rho'(g(\si))\in\big(\widetilde\Xi'\big)_{g(M)}^{g(N)}$\,. 

\smallskip
If ${\rm Sat}^\diamond\big[\rho'(g(\Si))\big]=\Xi'$, then every $\xi'\in\big(\widetilde\Xi'\big)_{g(M)}^{g(N)}$ can be written as 
$$
\xi'=\xi\diamond\rho'(g(\si))\,.
$$ 
Without lose of generality we may assume that $\si$ is in $M$ and satisfies $\xi'\bu'g(\si)\in g(N)$ (if not, select some $\si_0\in M$ such that $\xi'\bu'g(\si_0)\in g(N)$ and observe that $\d'(\xi')=\rho'(g(\si))=\rho'(g(\si_0))$). In this case, one has $\d(\xi)=\nu\big[\rho'(g(\si))\big]=\rho(\si)$\,, so 
$$
\xi'\bu'g(\si)=\big(\xi\diamond\rho'(g(\si))\big)\bu'g(\si)=g(\xi\bu\si)\in g(N)\,.
$$ 
Using the injectivity of $g$ one gets $\xi\bu\si\in N$, meaning that $\xi\in\Xi_{M}^N$\,, and the equality in \eqref{relashn} follows from the way $\xi'$ was defined.
\end{proof}

\begin{prop}\label{jnitzel}
\begin{enumerate}
\item[(i)]
Let $\si,\tau\in\Si$ such that $\si\overset{\bu}{\sim}\tau$. Then $g(\si)\overset{\bu'}{\sim}g(\tau)$\,. Consequently 
\begin{equation}\label{arab}
g\big(\mathfrak O_{\si}\big)\subset\mathfrak O_{g(\si)}'\quad{\rm and}\quad g\Big(\overline{\mathfrak O}_{\si}\Big)\subset\overline{\mathfrak O}_{g(\si)}'\,.
\end{equation}
\item[(ii)]
If $\,{\rm Sat}^\diamond\big[\rho'(g(\Si))\big]=\Xi'$, then 
\begin{equation*}\label{tropism}
g\big(\mathfrak O_{\si}\big)=\mathfrak O_{g(\si)}'\,.
\end{equation*}
Hence the function $g$ maps $\bu$-invariant subsets of $\,\Si$ to $\bu'$-invariant subsets of $\,\Si'$. 
\item[(iii)]
Inverse images through $g$ of $\bu'$-invariant subsets of $\,\Si'$ are $\bu$-invariant subsets of $\,\Si$\,.
\end{enumerate}
\end{prop}

\begin{proof}
(i) For $\si,\tau\in\Si$\,, if $\si\overset{\bu}{\sim}\tau$ then
\begin{equation*}\label{minat}
\begin{aligned}
\emptyset\ne\widetilde\Xi_{\si}^\tau\diamond\rho'(g(\si))\overset{\eqref{relashn}}{\subset}\big(\widetilde\Xi'\big)_{g(\si)}^{g(\tau)}\,,
\end{aligned}
\end{equation*}
which implies that $\big(\widetilde\Xi'\big)_{g(\si)}^{g(\tau)}\ne\emptyset$\,, i.e. $g(\si)\overset{\bu'}{\sim}g(\tau)$\,. In its turn, this implies \eqref{arab}. 

\smallskip
(ii) Pick $\tau'\overset{\bu'}{\sim}g(\si)$\,, related through $\tau'=\xi'\bu'g(\si)$\,. By the assumption ${\rm Sat}^\diamond\big[\rho'(g(\Si))\big]=\Xi'$, write 
$$
\xi'=\xi\diamond\rho'(g(\si_0))\,.
$$ 
We have 
$$
\d'(\xi')=\rho'(g(\si))=\rho'(g(\si_0))
$$ 
and thus we get 
$$
\tau'=\xi'\bu'g(\si)=\big[\xi\diamond\rho'(g(\si))\big]\bu'g(\si)=g(\xi\bu\si)\in g\big(\mathfrak O_{\si}\big)\,.
$$

(iii) This follows easily from \eqref{siete}. Let $B'\subset\Si'$ be $\bu'$-invariant. If
$$
\tau=\xi\bu\si\overset{\bu}{\sim}\si\in g^{-1}(B')\,,
$$
then
$$
g(\tau)=g(\xi\bu\si)=\big[\xi\diamond\rho'(g(\si))\big]\bu' g(\si)\in B'.
$$
\end{proof}

Clearly, what happens on $\Si'\setminus g\big(\Si\big)$ cannot be related to the action of $\Xi'$ on $\Si'$; this is particularly relevant for relating the different notions of transitivity. 

\begin{cor}\label{sentintaa}
Suppose that the map $g$ is surjective. 
\begin{enumerate}
\item[(i)] If $\,\Th$ is transitive, $\Th'$ is also transitive. If $\,\Th$ is pointwise transitive, $\Th'$ is also pointwise transitive.
\item[(ii)] If $\,\Th$ is weakly pointwise transitive, $\Th'$ is also weakly pointwise transitive.
\item[(iii)] If $\,\Th$ satisfy any of the properties $({\rm TT}_j)$\, with $j\in\{1,2\}$\,, then $\Th'$ also satisfy the same property. 
\item[(iv)] If $\,M\subset\Si$ is minimal and $g(M)\subset\Si'$ is closed, $g(M)$ is also minimal.
\end{enumerate}
\end{cor}

\begin{proof}
The proof only relies on Proposition \ref{jnitzel} and is almost identical to the proof of Corollary \ref{secinta}.
\end{proof}

\begin{cor}\label{sentinta}
Assume that the map $g$ is surjective. If $\,\Th$ is recurrently transitive, $\Th'$ is also recurrently transitive.
\end{cor}

\begin{proof}
Suppose that $\Th$ is recurrently transitive and let $\emptyset\ne U,V\subset\Si'$ two open sets. Then $g^{-1}(U),g^{-1}(V)\subset\Si$ are non-void open sets, so $\widetilde\Xi_{g^{-1}(U)}^{g^{-1}(V)}\ne\emptyset$\, which implies
$$
\emptyset\ne\widetilde\Xi_{g^{-1}(U)}^{g^{-1}(V)}\diamond\rho'(U)\subset\big(\widetilde\Xi'\big)_{U}^{V}\,.
$$ 
So $\Th'$ is recurrently transitive. 
\end{proof}

\begin{cor}\label{siaia}
Let $(\Phi,g):(\Xi,\rho,\bu,\Si)\rightsquigarrow\big(\Xi',\rho',\bu',\Si'\big)$ be an algebraic morphism of groupoid actions. Assume that $\,{\rm Sat}^\diamond\big[\rho'(g(\Si))\big]=\Xi'$. If $\si$ is periodic, $g(\si)$ is also periodic.
\end{cor}

\begin{proof}
If $\si$ is periodic, then $\widetilde\Xi_\si^\si$ is syndetic in $\Xi_{\rho(\si)}$ with compact set $\sf K$. By \eqref{fuame} and Theorem \ref{liema}, one has 
$$
\mathfrak O_{\rho'(g(\si))}^\diamond=\Xi_{\rho(\si)}\diamond\rho'(g(\si))=\big({\sf K}\widetilde\Xi_{\si}^\si\big)\diamond\rho'(g(\si))={\sf K}\diamond\Big[\widetilde\Xi_{\si}^\si\diamond\rho'(g(\si))\Big]\subset{\sf K}\diamond\big(\widetilde\Xi'\big)_{g(\si)}^{g(\si)}\,.
$$ 
Because of the Lemma \ref{constant}, we see that 
$$
\big(\Xi'\big)_{\rho'(g(\si))}=\mathfrak O_{\rho'(g(\si))}^\diamond\subset {\sf K}\diamond\big(\widetilde\Xi'\big)_{g(\si)}^{g(\si)}={\sf K}\diamond\Big[\rho'(g(\si))\big(\widetilde\Xi'\big)_{g(\si)}^{g(\si)}\Big]=\big[{\sf K}\diamond\rho'(g(\si))\big]\big(\widetilde\Xi'\big)_{g(\si)}^{g(\si)}\,.
$$ 
Hence $\big(\widetilde\Xi'\big)_{g(\si)}^{g(\si)}$ is syndetic in $\big(\Xi'\big)_{\rho'(g(\si))}$, with compact set ${\sf K'}={\sf K}\diamond\rho'(g(\si))$\,.
\end{proof}

\begin{cor}\label{siaia2}
Let $(\Phi,g):(\Xi,\rho,\bu,\Si)\rightsquigarrow\big(\Xi',\rho',\bu',\Si'\big)$ be an algebraic morphism of groupoid actions, with $\Si'$ a locally compact space. Assume that $\,{\rm Sat}^\diamond\big[\rho'(g(\Si))\big]=\Xi'$. If $\si$ is almost periodic, $g(\si)$ is also almost periodic.
\end{cor}

\begin{proof}
Assume that $\si$ is almost periodic, let $U'\subset \Si'$ be a neighborhood of $g(\si)$ and, without lose of generality, assume that $U'$ has a compact closure. Set $U=g^{-1}(U')$\,. Since $\si$ is almost periodic, we have $\Xi_{\rho(\si)}={\sf K}\widetilde\Xi_{\si}^U$ for some compact set ${\sf K}\subset \Xi$\,. We apply Theorem \ref{liema} just as in Corollary \ref{siaia} to find that 
$$
\mathfrak O_{\rho'(g(\si))}^\diamond={\sf K}\diamond\Big[\widetilde\Xi_{\si}^U\diamond\rho'(g(\si))\Big]\subset{\sf K}\diamond\big(\widetilde\Xi'\big)_{g(\si)}^{U'}\,.
$$ 
Because of Lemma \ref{constant}, we see that 
$$
\big(\Xi'\big)_{\rho'(g(\si))}=\mathfrak O_{\rho'(g(\si))}^\diamond\subset {\sf K}\diamond\big(\widetilde\Xi'\big)_{g(\si)}^{U'}=\big({\sf K}\diamond\rho'(U')\big)\big(\widetilde\Xi'\big)_{g(\si)}^{U'}\,.
$$ 
So we have ${\sf K'}\big(\widetilde\Xi'\big)_{g(\si)}^{U'}=\big(\Xi'\big)_{\rho'(g(\si))}$\,, with ${\sf K'}={\sf K}\diamond\rho'(\overline{U'})$ compact. The result follows. 
\end{proof}

When dealing with limit points, a properness notion is needed.

\begin{defn}\label{amu}
The algebraic morphism $\Phi:=(\Xi,\mu,\diamond,\Xi')$ from $\Xi$ to $\Xi'$ is called {\it proper} if the function
\begin{equation}\label{brober}
F^\diamond_{x'}:\Xi_{\nu(x')}\to\Xi'_{x'}\,,\quad F^\diamond_{x'}(\xi):=\xi\diamond x'
\end{equation}
is proper for every $x'\in X'\subset\Xi'$.
\end{defn}

\begin{ex}\label{ciucurel}
If $\Xi=\G$ and $\Xi'=\G'$ are groups, as in Example \ref{proposit}, there is a single interesting map
$$
F^\diamond_{\e'}:\G=\Xi_{\e}\to\G'=\Xi_{\e'}\,,\quad F^\diamond_{\e'}(\xi)=\xi\bu\e'=\beta(\xi)\e'=\beta(\xi)\,,
$$
so the algebraic morphism is proper if and only if the group morphism $\beta:\G\to\G'$ is proper.
\end{ex}

\begin{ex}\label{ciucurel2}
Suppose now that $\,\Xi=\bigsqcup_{x\in X}\G_x$ and $\,\Xi'=\bigsqcup_{x'\in X'}\G'_{x'}$ are group bundles, as in Example \ref{procopsit}. The groupoid units are precisely the group units of the fibres. For each $x'\in X'$ one has
$$
F^\diamond_{x'}\!:\G_{\nu(x')}\to\G'_{x'}\,,\quad F^\diamond_{x'}(\xi)=\xi\bu_{x'}\!\e_{x'}\!=\beta_{x'}(\xi)\e_{x'}=\beta_{x'}(\xi)\,.
$$
The algebraic morphism is proper if and only if all the group morphisms $\beta_{x'}$ are proper.
\end{ex}

\begin{ex}\label{obosit}
In Example \ref{poposit} we have seen that an algebraic morphism $\Phi=(X,\mu,\diamond,X')$ between two spaces (trivial groupoids) consists of a continuous function $\mu=\nu:X'\to X$ together with the action 
$$
\nu(x')\diamond x'=F^\diamond_{x'}\big(\nu(x')\big)=x'.
$$
The functions appearing in \eqref{brober} are constant maps defined between singletons, so the algebraic morphism is always proper.
\end{ex}

\begin{prop}\label{transflim}
Let $(\Phi,g):(\Xi,\rho,\bu,\Si)\rightsquigarrow\big(\Xi',\rho',\bu',\Si'\big)$ be an algebraic morphism of groupoid actions, with $\Phi$ proper. If $\si$ is a limit point for the first action, $g(\si)$ is a limit point for the second action. If $\si$ is recurrent for the first action, $g(\si)$ is recurrent for the second action. If $\si$ is a weakly periodic point for the first action, $g(\si)$ is a weakly periodic point for the second action.
\end{prop}

\begin{proof}
We rely on Remark \ref{precisely}. The beginning of the argument consists in copying the first part of the proof of Proposition \ref{caciu}. To finish the verification for limit points, one needs formula \eqref{relashn} and properness of the algebraic morphism. The statement about recurrent points is then a direct consequence.

\smallskip
Let $\tau'=g(\tau)$ with $\tau\in\mathfrak L_\si$ and let $V'$ be a neighborhood of $\tau'$. Then $V\!:=g^{-1}(V')$ is a neighborhood of $\tau$ and $g(V)\subset V'$. Since $\big(\widetilde{\Xi}'\big)_{g(\si)}^{g(V)}\!\subset\!\big(\widetilde{\Xi}'\big)_{g(\si)}^{V'}$\,, it is enough to show that $\big(\widetilde{\Xi}'\big)_{g(\si)}^{g(V)}$ is not relatively compact. Since $\tau$ is a limit point, $\widetilde{\Xi}_\si^V$ is not relatively compact. The inclusion \eqref{relashn} becomes here
\begin{equation}\label{relachn}
\widetilde\Xi_\si^V\diamond\rho'\big(g(\si)\big)\subset\big(\widetilde\Xi'\big)_{g(\si)}^{g(V)}\,.
\end{equation}
Suppose however that $\big(\widetilde{\Xi}'\big)_{g(\si)}^{g(V)}$ is relatively compact. Setting $x':=\rho'\big(g(\si)\big)$\,, we see by \eqref{fuame} that 
$$
\widetilde\Xi_\si^V\subset\Xi_{\rho(\si)}=\Xi_{\nu(x')}\,.
$$ Then, using the notation \eqref{brober},  
$$
\widetilde\Xi_\si^V\diamond\rho'\big(g(\si)\big)=F^\diamond_{\rho'(g(\si))}\big(\widetilde\Xi_\si^V\big)
$$ 
is also relatively compact. 
Since $\Phi$ has been assumed proper, $\widetilde\Xi_\si^V$ would also be relatively compact and we finished verifying the inclusion by contradiction. 

\smallskip
To prove the statement about weakly periodic points, repeat the same arguments, with $\{g(\si)\}$ instead of $V'$ and $\{\si\}$ instead of $V$.
\end{proof}

\section{Appendix. Some remarks on algebraic morphisms}\label{zgarier}

A particular case of the next proposition appears in \cite[Remark\;2.4, Prop.\;2.12]{Bu}. Our present result shows that, in a (very) special case, algebraic morphisms of groupoid actions can be reinterpreted as usual morphisms of actions. This also (greatly) extends Example \ref{terminalg}.

\begin{prop}\label{miraj}
Let $\,\Th:=(\Xi,\rho,\bu,\Si)$ and $\,\Th':=\big(\Xi',\rho',\bu',\Si'\big)$ two continuous groupoid actions. There is a one-to-one correspondence between
\begin{itemize}
\item algebraic morphisms of groupoid actions $\big(\Phi=(\Xi,\mu,\diamond,\Xi'),g\big)$ with $\nu:=\mu|_{X'}\!:X'\!\to X$ a homeomorphism,
\item morphisms of groupoid actions $(\Psi,f):\Th\to\Th'$ (as in Definition \ref{sprevest}) with $\psi\!:=\Psi|_X\!:X\to X'$ a homeomorphism. 
\end{itemize}
\end{prop}

\begin{proof}
Given $(\Psi,f)$\,, we set 
\begin{equation}\label{miramar}
g:=f,\quad\nu:=\psi^{-1},\quad\mu:=\nu\circ\r'\!=\psi^{-1}\!\circ\r',\quad\xi\diamond\xi'\!:=\Psi(\xi)\xi'.
\end{equation} 
The definition of the action $\diamond$ is justified by
$$
\d(\xi)=\mu(\xi')=\psi^{-1}\big[r'(\xi')\big]\Longrightarrow \psi[\d(\xi)]=\d'[\Psi(\xi)]=\r'(\xi')\,.
$$ 
Let us use \eqref{miramar} to check \eqref{siete}:
$$
\begin{aligned}
g(\xi\bu\si)&=f(\xi\bu\si)\\
&\overset{\eqref{ciudatta}}{=}\Psi(\xi)\bu'f(\si)\\
&=\Psi(\xi)\bu'\big[\rho'(f(\si))\bu'f(\si)\big] \\
&=\big[\Psi(\xi)\rho'(f(\si))\big]\bu'f(\si) \\
&\overset{\eqref{miramar}}{=}\big[\xi\diamond\rho'(f(\si))\big]\bu'f(\si)\\
&=\big[\xi\diamond\rho'(g(\si))\big]\bu'g(\si)\,.\\
\end{aligned}
$$
In the opposite direction, if $(\Phi,g)$ is given, set 
\begin{equation}\label{miramar2}
f\!:=g,\quad\psi:=\nu^{-1},\quad\Psi(\xi):=\xi\diamond\nu^{-1}[\d(\xi)]\,.
\end{equation} 
This is possible, as $\mu\big(\nu^{-1}[\d(\xi)]\big)=\d(\xi)$\,. We now check \eqref{ciudatta}, using \eqref{miramar2}: 
$$
\begin{aligned}
f(\xi\bu\si)&=g(\xi\bu\si)\\
&\overset{\eqref{siete}}{=}\big[\xi\diamond\rho'(g(\si))\big]\bu'g(\si)\\
&\overset{\eqref{fuame}}{=}\big[\xi\diamond\nu^{-1}(\rho(\si))\big]\bu'g(\si)\\
&=\big[\xi\diamond\nu^{-1}(\d(\xi))\big]\bu'g(\si)\\
&=\Psi(\xi)\bu'g(\si)\\
&=\Psi(\xi)\bu'f(\si)\,.
\end{aligned}
$$
\end{proof}

\begin{rem}\label{image} Under the situation of the previous Proposition, one finds out that 
$$
\Psi(\Xi)={\rm Sat}^\diamond(X')\,.
$$ 
Since $\rho'$ is assumed surjective, it follows that $(\Psi,f)$ is an epimorphism of actions if and only if $g$ is surjective and ${\rm Sat}^\diamond(X')={\rm Sat}^\diamond\big[\rho'(g(\Si))\big]=\Xi'$. We may think of the latter condition as being the equivalent for surjectivity in a generalized way. 
\end{rem}

We finish this appendix with two results on stabilizations under the action $\diamond$ resulting from an algebraic morphism of groupoids.

\begin{prop}\label{structure}
For every subset $Y'\!\subset X'$, ${\rm Sat}^\diamond(Y')$ is a subsemigroupoid of $\,\Xi'$ (stable for the multiplication). If $\,Y'$ is invariant ($\d'(\xi')\in Y'\!\Leftrightarrow\r'(\xi')\in Y'$)\,, ${\rm Sat}^\diamond(Y')$ is a subgroupoid. We have ${\rm Sat}^\diamond(Y')={\rm Sat}^\diamond(Z')$ if and only if $\,Y'=Z'$.
\end{prop}

\begin{proof}
If $\xi_1\diamond x_1', \xi_2\diamond x_2'\in {\rm Sat}^\diamond(Y')$ are elements such that $x_1'=\d'\big(\xi_1\diamond x_1'\big)=\r'\big(\xi_2\diamond x_2'\big)$\,, then 
$$
\big(\xi_1\diamond x_1'\big)\big(\xi_2\diamond x_2'\big)=\xi_1\diamond \Big[x_1'\big(\xi_2\diamond x_2'\big)\Big]=\xi_1\diamond \big(\xi_2\diamond x_2'\big)=(\xi_1\xi_2)\diamond x_2'\in {\rm Sat}^\diamond(Y')\,.
$$ 
Therefore, ${\rm Sat}^\diamond(Y')$ is closed under multiplication. Actually, the same computation proves the stronger property 
$$
{\rm Sat}^\diamond(X')\,{\rm Sat}^\diamond(Y')\subset{\rm Sat}^\diamond(Y')\,.
$$

Let us now assume that $Y'$ is invariant and prove that ${\rm Sat}^\diamond(Y')$ is closed by inversion: Let 
$$
\xi'=\xi\diamond\d'(\xi')\in {\rm Sat}^\diamond(Y')\,.
$$ 
We claim that 
$$
(\xi')^{-1}\!:=\xi^{-1}\!\diamond \r'(\xi')\in{\rm Sat}^\diamond(Y')\,.
$$ 
Indeed $\r'(\xi')\in Y'$, by invariance, and
$$
\mu(\r'(\xi'))=\mu(\xi')=\mu\big(\xi\diamond\d'(\xi')\big)=\r(\xi)=\d(\xi^{-1})
$$ 
proves that $\xi^{-1}$ may be applied to $\r'(\xi')$\,. Straightforward computations show that $\xi^{-1}\!\diamond \r'(\xi')$ is indeed the inverse of $\xi'$. For example
$$
\big(\xi^{-1}\!\diamond \r'(\xi')\big)\xi'=\xi^{-1}\!\diamond\Big[\r'(\xi')\xi'\Big]=\xi^{-1}\!\diamond \xi'=\xi^{-1}\!\diamond\big[\xi\diamond\d'(\xi')\big]=\d(\xi)\diamond \d'(\xi')=\d'(\xi')\,.
$$

Finally, if ${\rm Sat}^\diamond(Y')={\rm Sat}^\diamond(Z')$, then, for any $z'\in Z'$ there exists $y'\in Y'$ such that $z'=\xi\diamond y'$. But $z'$ is a unit. Apply \eqref{teans} to get $z'=y'\in Y'$, hence $Z'\!\subset Y'$. Then repeat the argument to conclude that $Y'\!=Z'$. 
\end{proof}

If our algebraic morphism comes from an algebraic morphism of actions, 
we can also show that ${\rm Sat}^\diamond\big[\rho'(g(\Si))\big]$ is a subgroupoid. This holds because we don't actually need the full strength of invariance. If $\r'(\xi')\in Y'$ holds for every $\xi'\in \Xi_{Y'}'$ of the form $\xi'=\xi\diamond y'$, the conclusion still follows. And this occurs in $\rho'(g(\Si))$ because of \eqref{siete}. 

\smallskip
As we said above, it is common to call a topological groupoid $\Xi$ {\it open} if its source map $\d:\Xi\to\Xi$ is open. Such a condition played in \cite{FM} a crucial role and it also appeared in the present article several times. It seems interesting to explore the fate of openness under algebraic morphisms. In the proof of the next result we will use Fell's criterion \cite[Prop.\,1.1]{Wi} for open functions twice.

\begin{prop}\label{enfin}
If $\,\Xi$ is an open groupoid and ${\rm Sat}^\diamond(X')=\Xi'$, then $\Xi'$ is an open groupoid.
\end{prop}

\begin{proof}
Let $x'=\d'(\xi')\in X'$ and pick some net $x'_i\in X'$ converging to $x'$. By the assumptions, we must have $\xi'=\xi\diamond x'$, for some $\xi\in\Xi$\,. Remark that $\mu\big(x'_i\big)\to\mu(x')=\d(\xi')$ so, by Fell's criterion, there exists a subnet $x'_{i_j}$ and a net $\xi_j\in\Xi$ such that $\mu\big(x'_{i_j}\big)=\d(\xi_j)$ and $\xi_j\to\xi$\,. Hence $\xi_j\diamond x'_{i_j}$ is well defined and 
$$
\xi_j':=\xi_j\diamond x'_{i_j}\to\xi\diamond x'=\xi'.
$$
So we found a subnet $x_{i_j}'\!=\d'(\xi'_j)$ with $\xi'_j\in\Xi$ converging to $\xi'$. We conclude, by applying Fell's criterion, that $\d'$ is an open map.
\end{proof}


{\bf Acknowledgements:} M. M. has been supported by the Fondecyt Project 1200884. F. F. has been supported by the Fondecyt Project 1171854. He is grateful to Aris Daniilidis for his kindness and help.



\medskip
M. Mantoiu:

\smallskip
Facultad de Ciencias, Departamento de Matem\'aticas, Universidad de Chile

Las Palmeras 3425, Casilla 653, 

Santiago, Chile.

E-mail: mantoiu@uchile.cl 

\medskip
F. Flores:

\smallskip
Departamento de Ingenier\'ia Matem\'atica, Universidad de Chile, 

Beauchef 851, Torre Norte, Oficina 436,

Santiago, Chile

E-mail: feflores@dim.uchile.cl

\end{document}